\documentclass[a4paper, 11pt]{amsart}
\usepackage{verbatim}
\usepackage{amsfonts}
\usepackage{amsmath}
\usepackage{amssymb}
\usepackage{amsthm}
\usepackage{enumerate}
\usepackage[utf8]{inputenc}
\usepackage{mathabx,epsfig}
\usepackage{tikz-cd}

\usepackage{xcolor}
\definecolor{green}{RGB}{0,127,0}
\definecolor{red}{RGB}{105,89,205}
\usepackage[colorlinks=true]{hyperref}

\usepackage{paralist, tabularx}
\usepackage[cmtip,arrow]{xy}

\DeclareMathOperator{\Aut}{Aut}
\DeclareMathOperator{\Hom}{Hom}
\DeclareMathOperator{\UT}{UT}
\DeclareMathOperator{\T}{T}
\DeclareMathOperator{\Stab}{Stab}
\DeclareMathOperator{\st}{st}

\DeclareMathOperator{\txtBohr}{{Bohr}}
\DeclareMathOperator{\txtdBohr}{{dBohr}}
\newcommand{\Bohr}[1]{{#1}^{\txtBohr}}
\newcommand{\dBohr}[1]{{#1}^{\txtdBohr}}

\DeclareMathOperator{\txtdef}{{def}}
\newcommand{\defi}[1]{{#1}_{\txtdef}}

\DeclareMathOperator{\tp}{{tp}}
\DeclareMathOperator{\cl}{{cl}}

\newcommand{\p}[1]{\widehat{#1}}

\newcommand{\R}{{\mathbb{R}}}
\newcommand{\Z}{{\mathbb{Z}}}
\newcommand{\N}{{\mathbb{N}}}
\newcommand{\Q}{{\mathbb{Q}}}
\newcommand{\F}{{\mathbb{F}}}
\newcommand{\ZZ}{{\bar{\Z}}}
\newcommand{\G}{{\bar{G}}}
\newcommand{\M}{{\bar{M}}}
\newcommand{\RR}{{\bar{R}}}

\newcommand{\Lsetp}[1]{\mathcal{L}_{set,{#1}}}
\newcommand{\LL}{\mathcal{L}}
\newcommand{\I}{\mathcal{I}}
\newcommand{\J}{\mathcal{J}}
\newcommand{\U}{\mathcal{U}}
\newcommand{\FF}[1]{\mathbb{F}_{#1}}

\newcommand{\Ra}{(R,+)}
\newcommand{\RRa}{(\RR,+)}
\newcommand{\RRaD}{(\RR,+)^{0}}
\newcommand{\RRaT}{(\RR,+)^{00}}
\newcommand{\RRaI}{(\RR,+)^{000}}
\newcommand{\RRiD}{\RR^{0}}
\newcommand{\RRiT}{\RR^{00}}
\newcommand{\RRiI}{\RR^{000}}

\newcommand{\ZZa}{(\ZZ,+)}
\newcommand{\ZZaD}{(\ZZ,+)^{0}}
\newcommand{\ZZaT}{(\ZZ,+)^{00}}

\newcommand{\ZZiT}{\ZZ^{00}}

\newcommand{\QM}[4]{  \begin{pmatrix} \begin{array}{c|c} {#1} & {#2} \\ \hline {#3} & {#4} \end{array}\end{pmatrix}}

\DeclareMathOperator{\topo}{{top}}

\theoremstyle{plain}
\newtheorem{theorem}{Theorem}
\numberwithin{theorem}{section}
\newtheorem{lemma}[theorem]{Lemma}

\newtheorem{fact}[theorem]{Fact}
\newtheorem{proposition}[theorem]{Proposition}
\newtheorem{remark}[theorem]{Remark}

\newtheorem{question}[theorem]{Question}
\newtheorem{corollary}[theorem]{Corollary}
\newtheorem{definition}[theorem]{Definition}

\newtheorem*{claim2}{Claim}

\theoremstyle{definition}
\newtheorem{example}[theorem]{Example}
\newtheorem*{remark2}{Remark}

\begin{document}

\title{Bohr compactifications of groups and rings}

\author{Jakub Gismatullin}

\address{(J. Gismatullin): Instytut Matematyczny Uniwersytetu Wroc{\l}awskiego, pl. Grunwaldzki 2/4, 50-384 Wroc{\l}aw, Poland \& Instytut Matematyczny Polskiej Akademii Nauk, ul. {\'S}niadeckich 8, 00-656 Warszawa, Poland \hspace{5mm}
ORCID: \href{https://orcid.org/0000-0002-4711-3075}{0000-0002-4711-3075}}

\email{jakub.gismatullin@uwr.edu.pl}

\thanks{\noindent The first author is supported by the Narodowe Centrum Nauki grants no.  2014/13/D/ST1/03491 and 2017/27/B/ST1/01467.}

\author{Grzegorz Jagiella}

\address{(G. Jagiella): Instytut Matematyczny Uniwersytetu Wroc{\l}awskiego, pl. Grunwaldzki 2/4, 50-384 Wroc{\l}aw, Poland \hspace{5mm}
ORCID: \href{https://orcid.org/0000-0002-5504-5260}{0000-0002-5504-5260}}

\email{grzegorz.jagiella@math.uni.wroc.pl}

\thanks{\noindent The second author is supported by the Narodowe Centrum Nauki grant no.  2018/31/B/ST1/00357.}

\author{Krzysztof Krupiński}

\address{ (K.\ Krupi\'{n}ski): Instytut Matematyczny Uniwersytetu Wroc{\l}awskiego, pl. Grunwaldzki 2/4, 50-384 Wroc{\l}aw, Poland \hspace{5mm}
ORCID: \href{https://orcid.org/0000-0002-2243-4411}{0000-0002-2243-4411}}

\email{Krzysztof.Krupinski@math.uni.wroc.pl}

\thanks{\noindent The third author is supported by the Narodowe Centrum Nauki grants no. 2016/22/E/ST1/00450 and 2018/31/B/ST1/00357.}


	
	\begin{abstract}
		We introduce and study model-theoretic connected components of rings as an analogue of model-theoretic connected components of definable groups. We develop their basic theory and use them to describe both the definable and classical Bohr compactifications of rings. We then use model-theoretic connected components to explicitly calculate Bohr compactifications of some classical matrix groups, such as the discrete Heisenberg group $\UT_3(\Z)$, the continuous Heisenberg group $\UT_3(\R)$, and, more generally, groups of upper unitriangular and invertible upper triangular matrices over unital rings.
	\end{abstract}
	
	\maketitle
	
	\tableofcontents
	
	\section*{Introduction}


The motivation for this research was the study of the \emph{model-theoretic connected components} of some matrix groups over unital rings in order to describe the classical \emph{Bohr compactifications} of these matrix  groups through the use of model theory.

Bohr compactifications of topological groups play an important role in topological dynamics and harmonic analysis, and they have some applications to differential equations. They allow to reduce many problems in the theory of almost periodic functions on topological groups to the corresponding problems about functions on compact groups. For example, see \cite{MR1120781, Pan}.
	
The model-theoretic connected components of a definable group $G$ (see Section \ref{section: preliminaries} for definitions) are among the fundamental objects used to study $G$ as a first-order structure. They are of particular significance in \emph{definable topological dynamics}, a generalization of classical topological dynamics. In \cite{GPP, KP}, the authors introduce and study the notion of the \emph{definable} Bohr compactification of a group $G$ definable in a first-order structure. This compactification is described in terms of one of the model-theoretic connected components of $G$. The classical Bohr compactification of a discrete group $G$ is a special case, and arises when $G$ is considered with the full set-theoretic structure (i.e. when every subset of $G$ is 0-definable). Also, the classical Bohr compactification of a topological group was described in  \cite{GPP, krupil} in terms of a suitably defined model-theoretic connected component.
	
	The calculation of model-theoretic connected components of matrix groups over a unital ring naturally led us to the development of the analogous notions of (model-theoretic) connected components of rings. These components were not studied so far and are interesting in their own right. 
In this paper, our first objective is to give precise definitions of various components of a ring (see Definition \ref{def:con} and the discussion following it), and prove some fundamental results about them such as Proposition \ref{prop:groupIdeal}, Proposition \ref{prop:con}, or Corollary \ref{cor:id}. In particular, we show in Proposition \ref{prop:con} that, as opposed to the group case, the appropriately defined $0$- and $00$-connected components of a unital ring always coincide. We also relate these components to the model-theoretic connected components of the additive group of the ring (see Corollaries \ref{cor:cor1_2}, \ref{cor:criteria}, the examples in Subsection \ref{subsection: rings vs groups}, and Proposition \ref{prop:R00Ideal}).
In Subsection \ref{subsection: Bohr comp of top. rings}, we observe that ring components can be used to describe the [definable] Bohr compactification of a discrete ring. 
In Subsection \ref{subsection:Bohr comp. top.}, we introduce a notion of a model-theoretic component for a topological ring and use it to describe the Bohr compactification of such a ring.  
Besides elementary algebraic and model-theoretic tools, also certain consequences of Pontryagin duality are involved in some arguments in the above part. All the  facts around Pontryagin duality which we need in this paper are discussed in the preliminaries.
	
Our original objective was to use 
model-theoretic connected components to explicitly compute both the definable and classical Bohr compactifications of some matrix groups. 
We focus on the groups $\UT_n(R)$ and $\T_n(R)$, where $R$ is a unital ring. We obtain a general description of the Bohr compactifications of these groups (see Propositions \ref{prop:utn00quot} and \ref{prop:tn00quot}). In the case of some classical rings, e.g. when $R$ is a field, or the ring of integers, or the ring of polynomials in several variables over an infinite field,
we get more precise descriptions, which in particular applies to the discrete Heisenberg group $\UT_3(\Z)$ (see general Corollary \ref{cor:utn00quot with dagger} and its applications in Subsection \ref{subsection: special rings}). We also adapt our approach to the groups $\UT_n(R)$ and $\T_n(R)$ treated as topological groups (for $R$ being a topological ring), obtaining descriptions of their classical Bohr compactifications, which in particular applies to the continuous Heisenberg group $\UT_3(\R)$ (see Propositions \ref{prop:utn00quot in topological case}, \ref{prop:tn00quottopo}, and Example \ref{example: top. field}).

Our method  of computing classical Bohr compactifications of the above matrix groups via model-theoretic connected components is novel, and, up to our knowledge, the descriptions of the Bohr compactifications which we obtained have not been known so far. 
	


As an example, let us state here our descriptions of the classical Bohr compactifications of both the discrete and continuous Heisenberg group.

The Bohr compactification of the discrete Heisenberg group $\UT_3(\Z)$ is 
	\[\begin{pmatrix}
			1 & \Bohr{\Z} & \hat{\Z}\\
			0 & 1 & \Bohr{\Z}\\
			0 & 0 & 1
			\end{pmatrix}:= 
\left\{ \begin{pmatrix}
			1 & a & b\\
			0 & 1 & c\\
			0 & 0 & 1
			\end{pmatrix} : a,c \in  \Bohr{\Z}, b \in \hat{\Z} \right\}, \]
where $\Bohr{\Z}$ is the Bohr compactification of the discrete group $(\Z,+)$, $\hat{\Z}$ is the profinite completion of $\Z$, and the product of two matrices from this set is defined as follows:

 \[\begin{pmatrix}
			1 & a & b\\
			0 & 1 & c\\
			0 & 0 & 1
			\end{pmatrix} 
 \begin{pmatrix}
			1 & \alpha & \beta\\
			0 & 1 & \gamma\\
			0 & 0 & 1
			\end{pmatrix} =
 \begin{pmatrix}
			1 & a+ \alpha& b+ \beta + \pi(a)\pi(\gamma)\\
			0 & 1 & c+\gamma\\
			0 & 0 & 1
			\end{pmatrix},\]
where $\pi \colon \Bohr{\Z} \to \hat{\Z}$ is a unique continuous group epimorphism compatible with the maps from $\Z$, provided by universality of  $\Bohr{\Z}$. More precisely, the Bohr compactification of $\UT_3(\Z)$ is the homomorphism from $\UT_3(\Z)$ to the above group of matrices which
is defined coordinatewise by the natural maps $\Z \to \Bohr{\Z}$ and $\Z \to \hat{\Z}$.

The Bohr compactification of the topological group $\UT_3(\R)$ is 
\[\begin{pmatrix}
			1 & \Bohr{\R} & 0\\
			0 & 1 & \Bohr{\R}\\
			0 & 0 & 1
			\end{pmatrix}:= 
\left\{ \begin{pmatrix}
			1 & a & 0\\
			0 & 1 & b\\
			0 & 0 & 1
			\end{pmatrix} : a,b \in  \Bohr{\R} \right\} \cong \Bohr{\R} \times \Bohr{\R}, \]
where $\Bohr{\R}$ is the Bohr compactification of the topological group $(\R,+)$. More precisely, the Bohr compactification of the topological group $\UT_3(\R)$ is the continuous homomorphism from  
$\UT_3(\R)$ to  $\Bohr{\R} \times \Bohr{\R}$ defined coordinatewise by the (Bohr compactification) map $\R \to \Bohr{\R}$.

	\section{Preliminaries}\label{section: preliminaries}
	In this paper, we use standard model-theoretic notations. We consider groups and rings as objects definable in some first-order structure $M$, and often assume the groups and rings themselves to be first-order structures in some language $\LL$ expanding the language of groups and rings, respectively.
	We always consider a structure $M$ together with a fixed $\kappa$-saturated and strongly $\kappa$-homogeneous elementary extension $\M \succ M$, where $\kappa > |M| + |\LL|$ is a strong limit cardinal. For a definable set $X \subseteq M$, we denote by $\bar{X}$ its interpretation in $\M$. We call a set $A \subseteq \M$ \emph{small} if $|A| < \kappa$. If $(G, \cdot, \ldots)$ is a definable group, we say that a subgroup $H \leq \G$ has \emph{bounded} index if $|\G/H| < \kappa$. For rings, by the \emph{index} of a subring we mean its index as a subgroup of the additive group of the ring.

Groups and rings will be often equipped with topology compatible with their operations. We make a blanket assumption that all topological groups, topological rings, and topological spaces in general that we consider in this paper are always Hausdorff (unless stated otherwise). We note however that most of the presented theory can be repeated, possibly with minor modifications, by requiring only compact spaces to be Hausdorff.

We will say that $M$ is equipped with the \emph{full (set-theoretic) structure} if all subsets of $M^{n}$, $n \in \N$, are 0-definable. 
The language of such a structure will be denoted by $\Lsetp{M}$. 
	
	If $G$ is a definable group in $M$, and $A\subset \bar M$ a small set, recall the following well-known subgroups of $\G$, so-called \emph{model-theoretic connected components}:
	\begin{compactitem}
		\item $\G^0_A$, the intersection of all $A$-definable subgroups of $\G$ with finite index,
		\item $\G^{00}_A$, the smallest $A$-type-definable subgroup of $\G$ with bounded index,
		\item $\G^{000}_A$, the smallest $A$-invariant subgroup of $\G$ with bounded index.
	\end{compactitem}
	We refer to \cite{modcon, BCG} for the properties of the connected components which we are going to use and explain below. Clearly $\G^{000}_A \leq \G^{00}_A \leq \G^0_A$. Sometimes (e.g. in theories with NIP), the group $\G^0_A$ does not depend on the choice of $A$, in which case we say that $\G^0 = \G^0_\emptyset$ exists, and similarly for the other components. Each component is a normal subgroup of $\G$. The quotients $\G/\G^{000}_A, \G/\G^{00}_A, \G/\G^{0}_A$ can be equipped with the \emph{logic topology} (where a set is closed if and only if its preimage under the quotient map is type-definable 
	over a small set of parameters), making them respectively a quasi-compact (i.e. not necessarily Hausdorff), a compact, and a profinite topological group. 
The same holds for the quotient of $\bar G$ by any normal subgroup of bounded index which is $A$-invariant, $A$-type-definable, or an intersection of some $A$-definable subgroups of $\G$, respectively.
	
	Let $G$ be a topological group. A \emph{compactification of $G$} is a compact topological group $K$ together with a continuous homomorphism $\phi\colon G \to K$ with dense image. The \emph{Bohr compactification of $G$} is a compactification $\phi\colon G \to K$ satisfying the following universal property: if $\phi'\colon G \to K'$ is a compactification of $G$, then $\phi' = f \circ \phi$ for a unique continuous homomorphism $f\colon K \to K'$. The Bohr compactification of $G$ always exists and is unique up to isomorphism. We will denote this object by $\Bohr{G}$. It is a classical notion in topological dynamics and harmonic analysis. It can be naturally extended to the category of topological rings, and other topological-algebraic objects, as done in \cite{Holm, HK}.
	
	The work in \cite{GPP} and \cite{KP} developed a model-theoretic version of Bohr compactifications. Let us briefly explain this setting. Suppose $X$ is a definable set and $C$ a compact topological space. Recall that a function $f \colon X \to C$ is said to be \emph{definable} if for each pair of disjoint, closed subsets $C_1, C_2 \subseteq C$, there are definable, disjoint subsets $U_1, U_2 \subseteq X$ such that $f^{-1}[C_i] \subseteq U_i$ for $i=1,2$. For a definable group $G$, we call its compactification $\phi \colon G \to K$ \emph{definable} if $\phi$ is definable. The results of \cite{GPP} show that a group $G$ definable in a model $M$ has the universal definable (called the \emph{definable Bohr}) compactification, which is just the quotient $\G/\G^{00}_M$ or rather the quotient homomorphism $G \to \G/\G^{00}_M$; 
we will denote it by $\dBohr{G}$. In the full (set-theoretic) setting $\Lsetp{M}$, $\dBohr{G}=\Bohr{G}$ (for $G$ treated as a discrete group), and the last result specializes to the following corollary.

	\begin{fact}\label{cor:gpp}
	Suppose $M$ is regarded in the language $\Lsetp{M}$ and $G$ is definable in $M$. Then $G \to \G/\G^{00}_M$ is the (classical) Bohr compactification of the discrete group $G$.
	\end{fact}
	 In this way, definable compactifications can be viewed as a generalization of classical ones.
	
	If $G$ is a locally compact abelian group, harmonic analysis provides a description of $\Bohr{G}$ 
	in terms of Pontryagin duality. Recall that the group $\Hom_c(G, S^1)$ of all continuous homomorphisms from $G$ into the circle group $S^1=\R/\Z=[-\frac{1}{2},\frac{1}{2})$ can be endowed with the compact-open topology, making it a locally compact abelian group. This object is called the \emph{Pontryagin dual} of $G$, which we denote by $\p{G}$: \[\p{G} = \Hom_c(G, S^1).\]

	\begin{fact}[\cite{Katz}, Chapter VII, Section 5]
		\label{fact:bohrPontryagin}
		Let $G$ be a locally compact abelian topological group. Then its Bohr compactification $\Bohr{G}$ is \[b\colon G\to  \Hom_c\left(\p{G}_{disc},S^1\right),\ \ g\mapsto (\varphi\mapsto \varphi(g)),\]
		where $\p{G}_{disc}$ denotes $\p{G}$ considered with the discrete topology. Moreover, still assuming that $G$ is a locally compact abelian group, the map $b$ is injective, that is for every $g\in G \setminus \{e\}$, there is $\varphi\in \p{G}$ such that $\varphi(g)\neq 0$.
	\end{fact}
	
	Observe that $\Hom_c\left(\p{G}_{disc},S^1\right) = \Hom\left(\p{G}_{disc},S^1\right)$, since $\p{G}_{disc}$ has the discrete topology.

	

The next fact is the famous Pontryagin duality (see \cite[Theorem 24.2, p. 376]{HR}).
\begin{fact}\label{fact: Pontryagin dulaity for LCA}
If $G$ is a locally compact abelian group, then $G \cong \p{\p{G}}$ as topological groups, via $b\colon G \to \p{\p{G}}$ given by the same formula as in Fact \ref{fact:bohrPontryagin}.
\end{fact}

 Recall that a \emph{profinite group} is an inverse limit of finite groups. The next fact is \cite[Theorem 2.9.6]{RZ}(b).

\begin{fact}
The Pontryagin dual of a profinite abelian group is a discrete, torsion abelian group. Conversely,
the Pontryagin dual of a discrete, torsion abelian group is a profinite abelian group.
\end{fact}

From the last two facts, we get:

\begin{corollary}\label{cor: consequence of Pontryagin duality useful for us}
A discrete abelian group $G$ is torsion if and only if $\p{G}$ is profinite.
\end{corollary}

    We will need the following fact (e.g. see \cite[Theorem 3.3.14]{dik}) in our analysis of model-theoretic connected components. We give a short proof based on Pontryagin duality.

	\begin{fact}
		\label{fact:profiniteBohr}
		A discrete abelian group $G$ is of finite exponent if and only if $\Bohr{G}$  
		is profinite.
	\end{fact}
	\begin{proof}
($\rightarrow$) Assume that  $G$ is of finite exponent. Then $\p{G}$ is also of finite exponent (as if $g^n=e$ for all $g \in G$, then $f^n(g)=f(g)^n=f(g^n) =0$ for any $f \in \p{G}$ and $g \in G$), so $\p{G}$ is a torsion abelian group. 
Therefore, by Fact \ref{fact:bohrPontryagin} and Corollary \ref{cor: consequence of Pontryagin duality useful for us}, $\Bohr{G}\cong \Hom\left(\p{G}_{disc},S^1\right)$ is profinite.

($\leftarrow$) Assume that  $\Bohr{G}$ is profinite. 
Then $\p{G}$ is torsion (again by Fact \ref{fact:bohrPontryagin} and Corollary \ref{cor: consequence of Pontryagin duality useful for us}). Suppose for a contradiction that $G$ is not of finite exponent. Then, by \cite[Lemma 4.9]{BCG}, there is a homomorphism from $G$ to $S^1$ with dense image. Such a homomorphism is an element of $\p{G}$ of infinite order, a contradiction.
%
	\end{proof}

\begin{remark2}
Alternatively, the implication ($\leftarrow$) can be obtained using the Baire category theorem and Fact \ref{fact: Pontryagin dulaity for LCA} in place of \cite[Lemma 4.9]{BCG}. First, observe that every torsion, compact abelian group $K$ has finite exponent. Indeed, by the Baire category theorem, for some $n \in \N_{>0}$ the closed subgroup $K[n]:=\{ k: k^n = e\}$ of $K$ is clopen and so of finite index; since $K$ is torsion and abelian, this implies that $K$ has finite exponent. 
Hence, since in our case $K: = \p{G}$ is torsion (by Fact \ref{fact:bohrPontryagin} and Corollary \ref{cor: consequence of Pontryagin duality useful for us}) and compact (by \cite[Proposition 2.9.1(b)]{RZ}), it has finite exponent. Therefore, since  $G \cong \p{\p{G}}$, we conclude that $G$ has finite exponent, too.
\end{remark2}

	By Fact \ref{fact:profiniteBohr} and Fact \ref{cor:gpp},
	we get the following

	\begin{corollary}
		\label{prop:bohrIfOnly}
		Let $G$ be an abelian group considered with the full structure. Then the Bohr compactification $\G/\G^{00}_G$ of $G$ is profinite if and only if $G$ has finite exponent.
	\end{corollary}
	\begin{corollary}
	If $G$ is an abelian group of finite exponent equipped with an arbitrary structure, then $\G/\G^{00}_A$ is profinite for any small set of parameters $A\subset \G$.
	\end{corollary}
	\begin{proof}
	We may assume that $A\subseteq G$. Let $\G$ be a monster model for both the full language and original language. Then $\G/\G^{00}_A$ is a topological quotient of $\G/\G^{00}_\emptyset$, where the former quotient is computed in the original language and the latter one in the full language. Since $\G/\G^{00}_\emptyset$ is a profinite group, so is $\G/\G^{00}_A$.
	\end{proof}
	

Another important consequence of Pontryagin duality is the following fact (see \cite[Proposition 5.1.2]{RZ}). Recall that under our assumptions, topological spaces are considered to be Hausdorff.

\begin{fact}\label{fact: profinite rings = compact rings}
A topological, unital ring is profinite if and only if it is compact.
\end{fact}
	Throughout the rest of the paper, $R = (R, +, \cdot, 0, 1, \ldots)$ is a (not necessarily commutative) unital ring, possibly with an additional structure, $\RR \succ R$ is a $\kappa$-saturated and strongly  $\kappa$-homogeneous elementary extension of $R$, and $A \subset \RR$ a small set of parameters. 
More generally, one can consider a ring $R$ which is $0$-definable in a structure $M$. 
We assume that $R$ is unital in order to proceed more smoothly in some proofs, apply Fact \ref{fact: profinite rings = compact rings}, or talk about the groups of invertible upper triangular matrices, but many definitions and observations work without this assumption, which will be mentioned in some places. However, an important consequence of Fact \ref{fact: profinite rings = compact rings}, namely the equality of the components in Proposition \ref{prop:con}(iv), requires unitality (see the discussion after Question \ref{quest: 000=00}).

Whenever we consider an ideal $I$ of a ring $R$, we will specify whether we mean a left, right, or two-sided ideal; except the cases where the (unital) ring $R$ is commutative.

	\section{Model-theoretic connected components of rings}

    \subsection{General theory}
	We define the following model-theoretic connected components of $\RR$ in a way analogous to model-theoretic connected components of groups. 
	\begin{definition}\phantomsection\label{def:con}
		\begin{enumerate}[(i)]
			\item $\RRiD_{A,ring}$ is the intersection of all $A$-definable subrings of $\RR$ with finite index.
			\item $\RRiT_{A,ring}$ is the smallest $A$-type-definable subring of $\RR$ with bounded index.
			\item $\RRiI_{A,ring}$ is the smallest $A$-invariant subring of $\RR$ with bounded index.
			\item $\RRiD_{A,ideal}, \RRiT_{A,ideal}, \RRiI_{A,ideal}$ are defined correspondingly using two-sided ideals instead of subrings.	
			\item We set $\RRiD_{A} := \RRiD_{A,ideal}$, $\RRiT_{A} := \RRiT_{A,ideal}$, $\RRiI_{A} := \RRiI_{A,ideal}$.
		\end{enumerate}
		If $\RRiD_{A,ring}$ does not depend on the choice of $A$, then we write $\RR^0_{ring} = \RR^0_{\emptyset, ring}$ and say that $\RR^0_{ring}$ exists. We do analogously for the remaining objects.
	\end{definition}


The existence of all these components (over the fixed set $A$) is clear and in fact the index of the smallest one (i.e. $\RRiI_{A,ring}$) in $\bar R$ is bounded by $2^{|\mathcal{L}| +|A|}$, as the relation of lying in the same coset of $\RRiI_{A,ring}$ is coarser than the finest bounded, $A$-invariant equivalence relation (i.e. the relation of having the same Lascar strong type over $A$) which is well-known to have at most $2^{|\mathcal{L}| +|A|}$ classes. 

	The following inclusions are obvious:
	\begin{center}
	\begin{tikzcd}
	\RRiI_{A,ring} \arrow[r, phantom, "\subseteq"] \arrow[d, phantom, sloped, "\subseteq"] & \RRiT_{A,ring} \arrow[r, phantom, "\subseteq"] \arrow[d, phantom, sloped, "\subseteq"] & \RRiD_{A,ring} \arrow[d, phantom, sloped, "\subseteq"]\\
	\RRiI_{A,ideal} \arrow[r, phantom, "\subseteq"] & \RRiT_{A,ideal} \arrow[r, phantom, "\subseteq"] & \RRiD_{A,ideal}
	\end{tikzcd}
	\end{center}
In fact, we prove in Proposition \ref{prop:con}(i)-(iii) that the components of the top row of the diagram coincide with the respective components of the bottom row. That is, there is no need to distinguish between the ring components and ideal components, which justifies item (v) of Definition \ref{def:con}. We moreover prove in Proposition \ref{prop:con}(iv) that the components $\RRiD_{A,ring}$ and $\RRiT_{A,ring}$ also coincide in any (unital) ring. This means that among the defined components there are only at most two distinct ones, and we leave as a question whether they coincide (see Question \ref{quest: 000=00}). We will keep distinguishing the components from the diagram until after Proposition \ref{prop:con} is proven.


The following example shows that (similarly to the group 00-component) the component $\RRiT_{R,ideal}$ can be thought of as a generalization of the kernel of the standard part map in the sense that it coincides with this kernel in a certain class of compact rings.
	\begin{example}
		If $R$ is a compact topological ring with a basis of neighborhoods of $0$ consisting of definable sets, and all definable subsets of $R$ have the Baire property, then $\RRiT_{R,ideal} = \ker(\st)$, where $\st \colon \RR \to R$ is the ``standard part'' map, and $\RR/\RRiT_{R,ideal} \cong R$. In particular,  this applies to the ring $\Z_p$ of $p$-adic integers in the (pure) language of rings.
	\end{example}

\begin{proof}
Let $\mu$ be the two-sided ideal of $\RR$ consisting of the infinitesimals, that is the intersection of the $\bar U$'s with $U$ ranging over all definable neighborhoods of $0$.
It is well-known that compactness of $R$ yields a well-defined group (in fact, also ring) homomorphism $\st \colon \RR \to R$ defined by $\st(r): =r'$ for a unique $r' \in R$ with $r-r' \in \mu$; moreover, $\ker(\st) = \mu$ and $\RR /\ker(\st) \cong R$ which is of bounded size. Therefore, $\RRiT_{R,ideal}\subseteq \ker(\st)$. It remains to show that $\ker(\st) \subseteq \RRiT_{R,ring}$. Write $\RRiT_{R,ring}$ as the intersection of some $R$-definable sets $\bar P_i$, $i \in I$, such that for every $i$ there is $j$ with $\bar P_j - \bar P_j \subseteq \bar P_i$. Then each $P_i$ (computed in $R$) is generic (that is some finitely many additive translates of $P_i$ cover $R$), and so, by compactness of $R$, each $P_i$ is non-meager. Since each $P_i$ has also the Baire property, we conclude from Pettis theorem (see \cite[Theorem 9.9]{Kec}) that each $P_i -P_i$ is a neighborhood of $0$. By the choice of the $P_i$'s, for every $i$ there is $j$ such that $P_j-P_j \subseteq P_i$, so we conclude that each $P_i$ is a definable neighborhood of $0$. Hence, $\ker(\st) = \mu \subseteq  \RRiT_{R,ring}$. 

Thus, we get the induced (abstract) isomorphism from $\RR/\RRiT_{R,ideal}$ to $R$. To see that it is a homeomorphism, it is enough to show that it is continuous (as both rings are compact). For this we need to check that $\st^{-1}[F]$ is type-definable for any closed subset $F$ of $R$. Note that $F=\bigcap_{r \in R \setminus F} U_r^c$ for some choice of definable neighborhood $U_r$ of $r$ (which exists by assumption). So it is enough to check that $\st^{-1}[F] = \mu + \bigcap_{r \in R \setminus F} \bar U_r^c$ which is clearly type-definable, where $\bar U_r$ is the interpretation of $U_r$ in $\bar R$. This is left for the reader.

The fact that the assumptions are satisfied for the ring $\Z_p$ follows from quantifier elimination in $\Q_p$ in Macintyre's language and the definability of $\Z_p$ in $\Q_p$ (see \cite{Bel}).
\end{proof}
	
		Consider the action of the \emph{monoid} $(\RR, \cdot)$ on $\RR$ by \emph{left} multiplication. For any $r \in \RR$, the map $f_r \colon \RR \to \RR$ given by $f_r(x) := r \cdot x$ is an endomorphism of the additive group $(\RR,+)$.
	For $X \subseteq \RR$, define its \emph{setwise stabilizer} as \[\Stab_\RR(X): = \left\{r \in \RR: r \cdot X \subseteq X\right\}.\] For any $X$, $\Stab_\RR(X)$ is a submonoid of $(\RR, \cdot)$. Moreover, if $X=G \leq (\RR,+)$ is a subgroup, then $\Stab_\RR(G)$ is a subring of $\RR$.
Further, if $X=S \subseteq \RR$ is a subring, then $\Stab_\RR(S)$ is the largest subring of $\RR$ in which $S$ is a left ideal; it is known as the \emph{left idealizer} of $S$ in $\RR$ \cite[p. 121]{good}.
	
	We similarly consider the action of $(\RR, \cdot)$ on $\RR$ by \emph{right} multiplication and denote the setwise stabilizer of $X$ under this action by $\Stab'_\RR(X):=\left\{r \in \RR: X \cdot r \subseteq X\right\}$.
	
	\begin{lemma}
		\label{lem:groupIdealLR}
		Let $G \leq (\RR, +)$ be a subgroup with bounded index such that $\Stab_\RR(G)$ [respectively $\Stab'_\RR(G)$] has bounded index.
		\begin{enumerate}[(i)]
			\item If $G$ is $A$-type-definable, then $G$ contains an $A$-type-definable left [respectively right] ideal of $\RR$ with bounded index.
			\item If $G$ is $A$-invariant, then $G$ contains an $A$-invariant left [respectively right] ideal of $\RR$ with bounded index.		
		\end{enumerate}
	\end{lemma}
	\begin{proof}
		(i) We give a proof for left ideals. Let $J(G) := \bigcap_{r \in \RR} f_r^{-1}[G]$; it is the largest left ideal of $\RR$ contained in $G$. We claim that $J(G)$ has the desired properties.
		It is clearly $A$-invariant. We may write $J(G)$ as $\bigcap_{r \in \RR} (f_r^{-1}[G] \cap G)$.
		\begin{claim2}
			If $r_1, r_2 \in \RR$ satisfy $r_1 - r_2 \in \Stab_\RR(G)$, then $f_{r_1}^{-1}[G] \cap G = f_{r_2}^{-1}[G] \cap G$.
		\end{claim2}
		\begin{proof}[Proof of Claim]
			It is sufficient to show that $f_{r_1}^{-1}[G] \cap G \subseteq f_{r_2}^{-1}[G]$. Write $r_2 = r_1 + a$ for some $a \in \Stab_\RR(G)$, and consider any $r' \in f_{r_1}^{-1}[G] \cap G$. Then $r_1r' \in G$ and $ar' \in G$, so $r_2 r' = r_1r' + ar' \in G$, and therefore $r' \in f_{r_2}^{-1}[G]$. This proves the claim.
		\end{proof}
		As $f_{r}^{-1}[G] \cap G$ depends only on the $\Stab_\RR(G)$-coset of $r$, $J(G)$ can be written as the intersection of a small number of type-definable sets over the same small set of parameters (namely $A$ together with a fixed set of representatives of the $\Stab_\RR(G)$-cosets), so $J(G)$ is type-definable. Since $J(G)$ is $A$-invariant, it is in fact $A$-type-definable. 
Since the subgroups  $f_{r}^{-1}[G]$, $r \in \RR$,  of the additive group of $\RR$ have uniformly bounded index,  an intersection of a small number of such subgroups is also a subgroup of bounded index. Hence, $J(G)$ has bounded index.

		(ii) follows by a similar argument.
	\end{proof}

A key point in what follows is the trivial observation below that the assumption that the index of the stabilizer is bounded is always satisfied when $G$ is a bounded index subring of $\RR$.

\begin{remark}\label{rem: trivial remark}
Let $S$ be a subring of $\RR$. Then $S \subseteq \Stab_\RR(S)$ and $S \subseteq \Stab'_\RR(S)$. Thus, if $S$ has bounded index, so do $\Stab_\RR(S)$ and $\Stab'_\RR(S)$.
\end{remark} 
	
	A standard observation about the connected components of groups is that each component has only boundedly many conjugates, so it must contain their intersection. In Lemma \ref{lem:groupIdealLR}, we instead used the assumption on the index of the stabilizer.
	%
Interestingly, the assumption that the index of the left [or right] stabilizer is bounded is sufficient to find a two-sided ideal instead of just one-sided one, as proved in the proposition below.
	\begin{proposition}\label{prop: two-sided ideal is contained}
		\label{prop:groupIdeal}
		Let $G \leq (\RR, +)$ be a subgroup with bounded index such that either $\Stab_\RR(G)$ or $\Stab'_\RR(G)$ has bounded index.
		\begin{enumerate}[(i)]
			\item If $G$ is $A$-type-definable, then $G$ contains an $A$-type-definable two-sided ideal of $\RR$ with bounded index.
			\item If $G$ is $A$-invariant, then $G$ contains an $A$-invariant two-sided ideal of $\RR$ with bounded index.	
		\end{enumerate}
	\end{proposition}


    \begin{proof}
(i)	Let $I_l$ and $I_r$ be the smallest $A$-type-definable left and right, respectively, ideals in $\RR$ with bounded index. By Remark \ref{rem: trivial remark} applied to $S:=I_l$, we see that $\Stab'_\RR(I_l)$ has bounded index. Thus, by Lemma \ref{lem:groupIdealLR}, $I_l$ contains $I_r$. In the same way, $I_r$ contains $I_l$. Hence, $I_l=I_r = \RRiT_{A,ideal}$ is a two-sided ideal.

Now, suppose that $\Stab_\RR(G)$ has bounded index (the case when  $\Stab'_\RR(G)$ has bounded index is similar). Then, by Lemma \ref{lem:groupIdealLR}, $G$ contains $I_l$, so we are done by the conclusion of first paragraph of this proof.		
		
(ii) The argument is again similar.
	\end{proof}

	We are now able to prove that some of the connected components introduced in Definition \ref{def:con} are actually equal.
	
	\begin{proposition}\phantomsection\label{prop:con}
		\begin{enumerate}[(i)]
			\item $\RRiI_{A,ring} = \RRiI_{A,ideal}$,
			\item $\RRiT_{A,ring} = \RRiT_{A,ideal}$, 
			\item $\RRiD_{A,ring} = \RRiD_{A,ideal}$,
			\item $\RRiD_{A,ring} = \RRiT_{A,ring}$.
		\end{enumerate}
	\end{proposition}
	
	\begin{proof}
		Items (i) and (ii) follow from Remark \ref{rem: trivial remark} and Proposition \ref{prop:groupIdeal} applied to $G=S:=\RRiI_{A,ring}$ and $G=S:=\RRiT_{A,ring}$, respectively.

We prove (iii) and (iv). Since the quotient ring $\RR/\RRiT_{A,ideal}$ is compact, it is profinite by Fact \ref{fact: profinite rings = compact rings}, so there is a basis of neighborhoods of $0$ that consists of clopen two-sided ideals. Let $\pi \colon \RR \to \RR/\RRiT_{A,ideal}$ be the quotient map. We have
		\[\RRiT_{A,ideal} = \bigcap\left\{\pi^{-1}[\I]: \I \text{ is a clopen two-sided ideal of } \RR/\RRiT_{A,ideal}\right\}.\]
		Consider a clopen two-sided ideal $\I$ of $\RR/\RRiT_{A,ideal}$. Both $\J :=\pi^{-1}[\I]$ and its complement are type-definable, hence definable. 
%
Also, $\J$ has finite index.
Since $\RRiT_{A,ideal} \leq \J$, the orbit of $\J$ under $\Aut(\RR/A)$ is bounded and so finite by definability of $\J$. Thus, $\bigcap_{f \in \Aut(\RR/A)} f[\J]$ is $A$-definable with finite index. 
This shows that $\RRiT_{A,ideal}$ is an intersection of $A$-definable two-sided ideals with finite index, and therefore 
$\RRiT_{A,ideal} = \RRiD_{A,ideal}$. In particular, $\RRiT_{A,ring} \subseteq \RRiD_{A,ring} \subseteq \RRiD_{A,ideal} = \RRiT_{A,ideal} = \RRiT_{A,ring}$ (where the last equality holds by (ii)), and so we get (iii) and (iv).
	\end{proof}
	We now adopt the notation from item (v) of Definition \ref{def:con}; that is, we write $\RRiI_A$ for $\RRiI_{A,ideal} (= \RRiI_{A,ring})$, $\RRiT_A$ for $\RRiT_{A,ideal} (= \RRiT_{A,ring})$, and $\RRiD_A$ for $\RRiD_{A,ideal} (= \RRiD_{A,ring})$.

	Proposition \ref{prop:con} establishes that $\RRiT_A = \RRiD_A$ regardless of the first-order structure of $R$. This is in stark contrast to the case of groups. The key difference is that due to Pontryagin duality, every (unital) compact topological ring is necessarily profinite, hence totally disconnected (which forces $\RR/\RRiT_A$ and $\RR/\RRiD_A$ to be the same object). The analogous statement is not true for groups. In particular, by Corollary \ref{prop:bohrIfOnly}, given an abelian group $G$ with infinite exponent considered with the full structure, the (compact) quotient $\bar G / \bar{G}^{00}_G$ is not profinite; it follows that $\bar{G}^{00}_G \neq \bar{G}^{0}_G$. A concrete instance of this case is $(\Z, +)$, discussed in more detail in Example \ref{ex:Z}. Another counterexample is the circle group $S := S^1(\mathbb{R})$ defined in an o-minimal expansion of $\mathbb{R}$. We have $\bar{S}^{0}_\emptyset = \bar{S}$, but $\bar{S}^{00}_\emptyset \neq \bar{S}^{0}_\emptyset$ as it consists of the infinitesimal elements of $\bar{S}$.
	
	Regarding the components $\RRiT_{A}$ and $\RRiI_{A}$, let us write explicitly what we have observed in the first paragraph of the proof of Proposition \ref{prop: two-sided ideal is contained}.
	
	\begin{corollary} \phantomsection\label{cor:id}
		\begin{enumerate}[(i)]
			\item $\RRiT_{A}$ is the smallest left and the smallest right $A$-type-definable ideal of $\RR$ with bounded index.
			\item $\RRiI_{A}$ is the smallest left and the smallest right $A$-invariant ideal of $\RR$ with bounded index.
		\end{enumerate}
	\end{corollary}
	

	\begin{question}\label{quest: 000=00} 
		Is $\RRiI_A = \RRiT_A \,(=\RRiD_A)$? Equivalently, is $\RR/\RRiI_A$ always profinite?
	\end{question}

This question is strongly related to some problems concerning our computation of the type-definable connected component of unitriangular groups, which will be discussed in Section \ref{sec:unig} after Question \ref{quest: finitely many steps in the sequence I_i}. In particular, see Lemma \ref{lem: conditions equivalent to 000=00} for equivalent statements.

We conclude this subsection with a discussion  on what happens if we drop the assumption that $R$ is unital. First, observe that this assumption is not needed in Lemma \ref{lem:groupIdealLR}, Remark \ref{rem: trivial remark} and Proposition \ref{prop:groupIdeal}
(working with $J(G) \cap G$ in place of $J(G)$ in the proof). However,
the assumption that $R$ is unital was used in the proofs of Proposition
\ref{prop:con} (iii) and (iv). 
Nevertheless, it turns out that (iii) holds also for non-unital rings, which is explained below, whereas (iv) fails in general: to see it, start from any abelian group $(R,+,\dots)$ for which $(\bar R,+)^{00}_A \ne (\bar R,+)^0_A$ and turn it into a (non-unital) ring with the trivial multiplication. Then the above additive group components coincide with the respective ring components, so $\RRiT_{A} = (\bar R,+)^{00}_A \ne (\bar R,+)^0_A= \RRiD_{A}$.

The proofs of Lemma \ref{lem:groupIdealLR}, Remark \ref{rem: trivial remark} and Proposition \ref{prop:groupIdeal} can be easily adapted to yield the following lemma.

\begin{lemma}\label{lemma: 0 for non-unital rings}
	Let $R$ be any (not necessarily unital) ring. Let $G \leq (\RR, +)$ be an $A$-definable subgroup with finite index
	such that $\Stab_\RR(G)$ [respectively $\Stab'_\RR(G)$] has finite index.
	Then:
	\begin{enumerate}[(i)]
		\item $G$ contains an $A$-definable left
		[respectively right] ideal of $\RR$ with finite index;
		\item if $S$ is a finite index subring of $\RR$,
		then $\Stab_\RR(S)$ and $\Stab'_\RR(S)$ are both of finite index;
		\item $G$ contains the intersection of all
		$A$-definable left ideals of finite index and also the intersection of all
		$A$-definable right ideals of finite index, and these two intersections
		coincide and form a two-sided ideal.
	\end{enumerate}
\end{lemma}

\begin{proof}
	(i) In the proof of Lemma \ref{lem:groupIdealLR}(i), it is enough to work with $J(G) \cap G$
	and observe that all $f_r^{-1}[G]\cap G$ are definable and of finite index
	and there are only finitely many of them.\\
	(ii) follows as in Remark \ref{rem: trivial remark}.\\
	(iii) We modify the proof of Proposition \ref{prop:groupIdeal}(i). Consider any
	$A$-definable left ideal $I$ of finite index. By (ii), $\Stab'_\RR(I)$ has
	finite index. Thus, by (i), $I$ contains an $A$-definable right ideal of
	finite index. Symmetrically, we have the same statements for switched roles
	of ``left'' and ``right''. This implies that the intersection of all
	$A$-definable left ideals of finite index coincides with the intersection
	of all $A$-definable right ideals of finite index, and so it is a
	two-sided ideal. Moreover, by (i) this two-sided ideal is contained in $G$.
\end{proof}

\begin{proposition}
	For an arbitrary (not necessarily unital)  ring $R$, $\RRiD_{A,ring} =
	\RRiD_{A,ideal}$ coincides with the intersection of all $A$-definable left
	[right] ideals of finite index.
\end{proposition}

\begin{proof}
	By Lemma \ref{lemma: 0 for non-unital rings}, the intersection of all
	$A$-definable left [right] ideals of finite index is a two-sided ideal $I$. Since $I$ is type-definable, $\bar R/ I$ is a compact topological ring (with the
	logic topology). It is also profinite as a group, as $I$ is an
	intersection of definable finite index subgroups.
	\begin{claim2}
	If a topological ring is profinite as a group, then it is profinite as a ring. In particular, $\bar R/I$ is profinite as a ring.
	\end{claim2}
	\begin{proof}[Proof of Claim.]
Let $S$ be a topological ring which is profinite as a group. Then $S$ has a basis of neighborhoods of $0$ consisting of clopen subgroups, and we need to show that it has a basis of neighborhoods of $0$ consisting of clopen two-sided ideals. So take a clopen subgroup $V \subseteq S$. For each $x \in S$, there are open neighborhoods $U_x \ni 0$ and $V_x \ni x$ such that $V_x U_x V_x + U_x \subseteq V$. By compactness, there are finitely many $x_0, x_1, \ldots, x_{n-1}$ such that $S = \bigcup_{i < n}V_{x_i}$. Put $U = \bigcap_{i < n}U_{x_i} $. Clearly, $U$ is an open neighborhood of $0$ and $S U S + U \subseteq V$. Let $H$ be the group generated by $S U S + U$. Then $H$ is a two-sided ideal. Since $S U S + U$ is open, $H$ is open (therefore clopen), and $H \subseteq V$ because $V$ is a group. This suffices.
	\end{proof}
	
	Hence, as in the proof of Proposition
	\ref{prop:con},
	we get that $I$ is an intersection of $A$-definable two-sided ideals of
	finite index. Thus, $\RRiD_{A,ideal} \subseteq I$, but the opposite
	inclusion is immediate from the definition of $I$, so we have equality.
	Hence, by Lemma \ref{lemma: 0 for non-unital rings} (ii) and (iii), we
	easily get  $\RRiD_{A,ideal} \subseteq  \RRiD_{A,ring}$, while the
	opposite inclusion is obvious.
\end{proof}

\subsection{Characterization of the ring components}
We now give a characterization of the ring components in terms of subgroups of the additive group. For convenience, the following result is stated in two parts, even though the components $\RRiT_A$ and $\RRiD_A$ are equal.
	\begin{proposition} \label{prop:charComp}
	\begin{enumerate}[(i)]
			\item $\RRiT_A$ is the intersection of all $A$-type-definable subgroups $G$ of $\RRa$ with bounded index such that \mbox{$[\RRa : \Stab_{\RR}(G)]$} is bounded.
			\item $\RRiD_A$ is the intersection of all $A$-definable subgroups $G$ of $\RRa$ with finite index such that \mbox{$[\RRa : \Stab_{\RR}(G)]$} is finite.
	\end{enumerate}
	\end{proposition}
	\begin{proof}
		If $G$ is a subgroup of $\RRa$ with bounded index such that  $\Stab_{\RR}(G)$ has bounded index, then, by Proposition \ref{prop:groupIdeal}(i), $\RRiT_A \subseteq G$, and so:
	\begin{align*}
		\RRiT_A & \subseteq \bigcap \left\{G \leq \RRa : G \text{ is } A\text{-type-definable}, [\RR : G] < \kappa, [\RR : \Stab_{\RR}(G)] < \kappa\right\} \\
		& \subseteq \bigcap \left\{G \leq \RRa : G \text{ is } A\text{-definable}, [\RR : G] < \omega, [\RR : \Stab_{\RR}(G)] < \omega\right\} \\
		& \subseteq \bigcap \left\{I \leq \RRa : I \text{ is an } A\text{-definable two-sided ideal with finite index}\right\} \\
		& = \RRiD_A = \RRiT_A. \qedhere
		\end{align*}
	\end{proof}
	Conversely, we have the following lemma.
	\begin{lemma} Let $G$ be an $A$-type-definable subgroup of $\RRa$ with bounded index. The following conditions are equivalent:
	\label{lem:TDcontain}
	\begin{enumerate}[(i)]
			\item $\RRiT_A \subseteq G$,
			\item $\RRiT_A \subseteq \Stab_\RR(G)$,
			\item $\Stab_\RR(G)$ has bounded index.
	\end{enumerate}
		Also, we can replace $\Stab_\RR(G)$ with $\Stab'_\RR(G)$ in items (ii) and (iii).
	\end{lemma}
	\begin{proof}
		For (i) $\rightarrow$ (ii): as $\RRiT_A$ is a two-sided ideal, we have $\RRiT_A \cdot G \subseteq \RRiT_A \subseteq G$. The implication (ii) $\rightarrow$ (iii) is immediate, and (iii) $\rightarrow$ (i) follows from Proposition \ref{prop:groupIdeal}(i).
	\end{proof}
	Likewise the lemma below for $A$-definable groups.
	\begin{lemma} Let $G$ be an $A$-definable subgroup of $\RRa$ with finite index. The following are equivalent:
	\label{lem:Dcontain}
	\begin{enumerate}[(i)]
			\item $\RRiD_A \subseteq G$,
			\item $\RRiD_A \subseteq \Stab_\RR(G)$,
			\item $\Stab_\RR(G)$ has finite index.
	\end{enumerate}
		Also, we can replace $\Stab_\RR(G)$ with $\Stab'_\RR(G)$ in items (ii) and (iii).
	\end{lemma}
	\begin{proof}
		This follows from the preceding lemma, because $\RRiD_A = \RRiT_A$, and if $G$ is $A$-definable,  then $\Stab_\RR(G)$ is also $A$-definable, and so $\Stab_\RR(G)$ is of bounded index if and only if it is of finite index.
	\end{proof}
	\begin{corollary}
		Let $G \leq \RRa$ be $A$-type-definable with bounded index. Then $\Stab_\RR(G)$ has bounded index if and only if $\Stab'_\RR(G)$ has bounded index. If $G$ is $A$-definable, then the same holds for ``bounded'' replaced by ``finite''.
	\end{corollary}
	
	\subsection{Ring components vs. additive group components}\label{subsection: rings vs groups}
	
	
	Our goal is to compare the connected components of $\RR$ to the connected components of the additive group $(\RR,+)$. We start with an immediate observation.
	
	\begin{remark}\label{rem: basic inclusions between  components}
		$\RRaT_A \leq \RRaD_A \leq \RRiD_A = \RRiT_A$.
	\end{remark}
	Therefore, if $\RRiT_A = \RRaT_A$, then $\RRaD_A = \RRaT_A$. It is natural to ask under which conditions $\RRiT_A$ is equal to one of the group components. Namely,
	\begin{question}
	When $\RRiT_A = \RRaT_A$? When $\RRiT_A = \RRaD_A$?
	\end{question}
	

	Our objective is now to find a characterization of when $\RRaD_A = \RRaT_A$. This equality means exactly that the group quotient $\RR/\RRaT_A$ is profinite (this equivalence is well-known and can be justified by an argument as in the proof of Proposition \ref{prop:con}). Below is an immediate consequence of Corollary \ref{prop:bohrIfOnly} for additive groups of rings.
	\begin{corollary}\label{cor:cor1_2}
		Suppose that $R$ is considered with the full structure.
		\begin{enumerate}[(i)]
			\item If $\Ra$ has infinite exponent, then $\RRa/\RRaT_R$ is not profinite, and so \[\RRiT_R \neq \RRaT_R \neq \RRaD_R.\]
			\item If $\Ra$ has finite exponent, then $\RRa/\RRaT_R$ is profinite, and so $\RRaT_R = \RRaD_R$.
		\end{enumerate}
	\end{corollary}


	A fundamental example of a ring whose additive group has infinite exponent is the ring of integers. Regardless of the structure on $\Z$, every subgroup of $\ZZa$ with finite index is of the form $n\ZZ$ for some $n \neq 0$, and so it is $0$-definable. Hence, for any structure on $\Z$, $\ZZaD = \bigcap_{n \neq 0}n\ZZ$ exists and is an ideal, so it coincides with $\bar\Z^0$ (which therefore exists).

	\begin{example}
		\label{ex:Z}
		Consider $\Z$ with the full structure. Since $\Z$ has infinite exponent, the above comment and Corollary \ref{cor:cor1_2} imply that $\ZZiT_\Z = \bar \Z^0 = \ZZaD \neq \ZZaT_\Z$.
	\end{example}

Using more explicit arguments, in \cite[Example 4.5]{BCG} the same conclusion was obtained working with the pure ring structure $(\Z,+,\cdot)$.

The core argument behind Corollary \ref{prop:bohrIfOnly} relies on harmonic analysis and the description of the Bohr compactification which it provides. On the other hand, both this corollary as well as the corollaries which we derive from it are stated in algebraic and model-theoretic terms. This leads to a question whether they can be proved by means of model-theory, e.g.:
	%
	\begin{question}
		Can one prove Corollary \ref{cor:cor1_2} without referring to Pontryagin duality?
	\end{question}
	
	We have already seen that $\RRiT_A = \RRiD_A$ may be strictly bigger than $\RRaT_A$. Now, we give examples where $\RRiT_A$ is strictly bigger than $\RRaD_A$.
	
	\begin{example}
	We are going to find an infinite field $K$ and a 0-definable proper subgroup $H < R=(\bar{K},+)$ with finite index. In a field of characteristic $p > 0$, such a subgroup always exist, and we can add a predicate for it. But we also give  an example for a pure field structure. Let $p$ be prime and $n \in \N_{>0}$. Consider the finite field $\FF{p^n}$ in the language of rings. The $0$-definable function $f \colon \FF{p^n} \to \FF{p^n}$ given by $f(x) = x^p - x$ is a homomorphism of $(\FF{p^n},+)$ whose kernel is the prime field $\FF{p} \subseteq \FF{p^n}$. Hence, the image of $f$ is a subgroup of $(\FF{p^n},+)$ with index $p$, and this is also true in the ultraproduct $K:=\prod_{n\in\N} \F_{p^n}/\U$ for a non-principal $\U$. Then $K$ is infinite and has the desired subgroup $H$.
Then $\RRaD_A \leq H \subsetneq \RRiT_A = \bar{K}$.
	\end{example}
	
	\begin{remark}
		In a field $K$ of characteristic $0$, the group $(K,+)$ is divisible and has no subgroups of finite index, so $(\bar{K},+)^0 = \bar{K}$ exists and coincides with $\bar K^{00}$.
	\end{remark}

     	Lemmas \ref{lem:TDcontain} and \ref{lem:Dcontain} give us the following straightforward criteria for when the type-definable connected component of $\RR$ differs from the connected components of $(\RR,+)$.
	
	\begin{corollary}
		\phantomsection\label{cor:criteria}
		\begin{enumerate}[(i)]
			\item $\RRiT_A \neq \RRaT_A$ if and only if there exists an $A$-type-definable $G \leq \RRa$ with bounded index such that \mbox{$[\RRa : \Stab_{\RR}(G)]$} is unbounded.
			\item $\RRiD_A \neq \RRaD_A$ if and only if there exists an $A$-definable $G \leq \RRa$ with finite index such that \mbox{$[\RRa : \Stab_{\RR}(G)]$} is infinite.
		\end{enumerate}
	\end{corollary}

Observe that if $A \subseteq R$, then on the right-hand side of the second criterion the ring $\RR$ can be replaced by $R$.

Now, we give an application of the second criterion. The example $\Z_2[X]$ was suggested to us by Światosław Gal.
	\begin{example}\label{example: ring 0 not equal group 0}
(1) Let $R:= \Z_2[X]$ be equipped with the full structure. We will show that it satisfies the right hand side of Corollary \ref{cor:criteria}(ii) for any $A\subseteq R$, so $\RRiD_A \neq \RRaD_A$.
Let $h\colon R \to \Z_2$ be given by
\[h\left(\sum a_i X^i\right) := \sum_{ k \in \omega} a_{2^k}.\]
Then $h$ is an epimorphism of groups and $G := \ker h$ is a subgroup of $R$ of index 2.
We will check that $f \in \Stab_{R}(G)$ iff $f$ is constant, which directly implies that $[\RRa : \Stab_{\RR}(\bar G)]$ is infinite.

Clearly $0, 1 \in \Stab_{R}(G)$. Now, take $f \in \Z_2[X]$ with $\deg(f) = k > 0$. Fix some natural $n > 1$ such that $2^n - k > 2^{n-1}$. Let $g := X^{2^n - k}$. Then $g \in G$, but $h(f \cdot g) = a_{2^n} = 1$, so $f\cdot g \notin G$.

(2) The above example generalizes to any $R:=K[\bar X]$ equipped with the full structure, where $K$ is a field of characteristic $p>0$ and $\bar X =(X_i)_{i<\lambda}$ is a (possibly infinite) tuple of variables. Namely, let $h\colon R \to \Z_p$ be given by
\[h\left(\sum a_{\bar i} X^{\bar i}\right) := \sum_{k \in \omega} \pi(a_{2^k}),\]
where $\pi \colon (K,+) \to (\Z_p,+)$ is any group homomorphism which is the identity on $\Z_p$, and $a_{2^k}$ is $a_{\bar i}$ for the tuple $\bar i$ with $2^k$ on the first position and $0$ elsewhere. As in (1), $G:=\ker(h)$ has finite index in $(R,+)$, whereas $\Stab_R(G)$  has infinite index, because each polynomial in $K[X_0] \setminus \Z_p$  is not in $\Stab_R(G)$. So $\RRiD_A \neq \RRaD_A$ for any $A\subseteq R$.
   \end{example}

Example \ref{example: ring 0 not equal group 0}(1) implies that for $R :=\Z[X]$ we also have $\RRiD_A \neq \RRaD_A$, but in order to see this, we need to make a few general remarks which may be useful in other situations, too. 

\begin{remark}\label{remark: f[00]=00}
Suppose $R,S$ are rings $A$-definable in some structure $M$ and $f \colon S \to R$ is an $A$-definable epimorphism. Then $f[\bar S^*_A] =R^*_A$ and $f[(\bar S,+)^*_A]= (\bar R,+)^*_A$, where $* \in \{0,00,000\}$.
\end{remark}

\begin{proof}
This follows easily from the fact that for any group epimorphism $h \colon G \to H$ and subgroups $K \leq G$ and $L \leq H$, we have $[H:f[K]] \leq [G:K]$ and $[G:f^{-1}[L]] \leq [H:L]$. 
\end{proof}

Notice that whenever $R$ is a ring definable in a structure $M$, then each of the components $\bar R^*_A$ and $(\bar R,+)^*_A$ (where $*\in \{0,00,000\}$ and $A \subseteq R$) computed with respect to the language $\mathcal{L}_{set,R}$ coincides with the one with respect to the language $\mathcal{L}_{set,M}$.
\begin{example}
For $S:=\Z[\bar X]$ ($\bar X$ a tuple of variables of an arbitrary length) equipped with the full structure and any $A \subseteq S$ we have $\bar S^0_A \neq (\bar S,+)^0_A$. 

In order to see this, let $R:=\Z_2[X]$ and take an epimorphism $f \colon S \to R$. Let $M$ consist of two sorts $S$ and $R$ and equip it with the full structure. 
By the comment preceding this example, we can compute our components with respect to   $\mathcal{L}_{set,M}$ in place of  $\mathcal{L}_{set,R}$. Since $f$ is 0-definable in $M$, the conclusion follows from Example \ref{example: ring 0 not equal group 0}(1) and Remark \ref{remark: f[00]=00}.
\end{example}

In Example \ref{ex:Z}, the left hand side of the criterion in \ref{cor:criteria}(i) holds, so the right hand side holds as well. But can one see directly that the RHS of (i) holds in this example? Also, the left hand of the criterion in \ref{cor:criteria}(ii) fails in this example, so the right hand side fails as well, but this is trivially seen directly, as each subgroup of finite index of $(\Z,+)$ is an ideal.

	Below we show a positive result for the case of a group component which does not depend on the parameters (which is for example always the case under NIP).

	\begin{proposition}
		\label{prop:R00Ideal}
		Let $(R, +, \cdot, 0, 1, \ldots)$ be a (unital) ring.
		\begin{enumerate}[(i)]
			\item If $\RRaD$ exists, then $\RRiT = \RRiD$ exists and $\RRaD = \RRiT$.
			\item If $\RRaT$ exists, then $\RRiT$ exists and $\RRaT = \RRaD = \RRiT$.
			\item If $\RRaI$ exists, then $\RRiI$ exists and $\RRaI = \RRaT = \RRaD = \RRiT = \RRiI $.
		\end{enumerate}
	\end{proposition}
	\begin{proof}
		By Corollary \ref{cor:id} and Remark \ref{rem: basic inclusions between  components}, to prove (i), it is sufficient to show that if $\RRaD$ exists, then it is a left ideal. For any $r \in \RR$, the set $f_r^{-1}[\RRaD]$ is an intersection of definable subgroups of $\RRa$ of finite index, so \mbox{$\RRaD \subseteq f_r^{-1}[\RRaD]$}.

In (ii), the proof of $\RRaT = \RRiT$ is similar; then the remaining equality follows from Remark \ref{rem: basic inclusions between  components}.

In (iii), the proof of  $\RRaI = \RRiI $ is again similar. Since $\RRaI$ exists, so does $\RRaT$. As $(\RR,+)$ is abelian, by \cite[Theorem 0.5]{krupil}, we have  $\RRaI = \RRaT$. The remaining equalities follow from (ii).
	\end{proof}

	\subsection{Definable compactifications of rings}\label{subsection: Bohr comp of top. rings}
	
	We now turn our attention to the notion of definable compactifications of rings. Let us recall the notion of \emph{definable compactification}.
	
	\begin{definition} \label{defin}
		\begin{enumerate}
		    \item For a definable $X\subseteq R$ and a compact topological space $C$ a function $f \colon X \to C$ is said to be \emph{definable} if for each pair of disjoint, closed subsets $C_1, C_2 \subseteq C$, there are definable, disjoint subsets $U_1, U_2 \subseteq X$ such that $f^{-1}[C_i] \subseteq U_i$ for $i=1,2$.
			\item A definable compactification of a ring $R = (R, +, \cdot, \ldots)$ is a compact topological ring $C$ together with a definable ring homomorphism $\phi \colon R \to C$ with dense image.
			
			\item The definable Bohr compactification of $R$ is a unique up to isomorphism definable compactification $\phi \colon R \to C$ which satisfies the following universal property: if $\phi'\colon R \to C'$ is a definable compactification of $R$, then $\phi' = f \circ \phi$ for a unique continuous homomorphism $f\colon C \to C'$.
		\end{enumerate}
	\end{definition}
	As in the context of groups, if a ring $R$ is considered in the full set-theoretic language $\Lsetp{R}$, then a definable [Bohr] compactification is the same thing as a classical [resp. Bohr] compactification of $R$ considered with the discrete topology.
	
	An essential result of \cite{GPP} shows the existence and uniqueness of the definable Bohr compactification of a definable group by means of its connected components. We state an analogous result for rings.

	\begin{proposition}\label{prop: def bohr comp. of R}
	The definable Bohr compactification of $R$ is $\RR/\RRiT_R$ with the natural map 
		\[R \ni r \mapsto r/\RRiT_R \in \RR/\RRiT_R.\] 
	\end{proposition}

	\begin{proof}
		This is proven similarly to Proposition 3.4 of \cite{GPP}, as the argument about lifting a group homomorphism also works for ring homomorphisms. 
	\end{proof}

The above definition and proposition are valid also for non-unital rings.
Let us note that in contrast with groups, by Fact \ref{fact: profinite rings = compact rings}, the definable Bohr compactification of a unital ring $R$ coincides with the universal definable, profinite compactification of $R$.


\subsection{Model-theoretic connected components of topological rings and the classical Bohr compactification}\label{subsection:Bohr comp. top.}

In \cite{GPP} and in Section 2 of \cite{krupil}, the classical Bohr compactification of a topological group $G$ was described as $\bar G/{\bar G}^{00}_{\topo}$, where $G$ is equipped with a structure in which all open sets are 0-definable (e.g. with the full structure), where ${\bar G}^{00}_{\topo}$ can be described as the smallest 0-type-definable, bounded index subgroup of $\bar G$ containing the infinitesimals. In fact, several equivalent definitions of ${\bar G}^{00}_{\topo}$ are given in  Section 2 of \cite{krupil}.

Here, we want to present an analog for topological rings, describing their Bohr compactifications (which coincide with the universal profinite compactifications for unital rings) in terms of a suitable component, where the {\em Bohr compactification} of a topological ring $R$ is, of course, defined as the unique universal (ring) compactification of $R$.

Let $R$ be a ring 0-definable in a structure $M$ so that all open subsets of $R$ are 0-definable (e.g. $M=R$ is equipped with the full structure). Let $\mu$ be the ring of infinitesimal elements in $\bar R$. Then $R\mu \subseteq \mu$, but $\mu$ is not necessarily a left ideal of $\bar R$. 

\begin{lemma}\label{lemma: mu + 00 is an ideal}
$\mu + \RRiT_M$ is a two-sided ideal of $\bar R$.
\end{lemma}
\begin{proof}
Put $G:= \mu + \RRiT_M$, an $M$-type-definable subring of bounded index. As before, for any $r \in \RR$, let $f_r \colon \RR \to \RR$ be given by $f_r(x):=r \cdot x$. Let  $J(G) := \bigcap_{r \in \RR} f_r^{-1}[G]$ (intersected additionally with $G$, if one wishes to drop our general assumption that $R$ is unital); it is the largest left ideal of $\RR$ contained in $G$. By Remak \ref{rem: trivial remark} and the proof of Lemma \ref{lem:groupIdealLR}, we can find a small set $S$ (e.g. a set of representatives of cosets of $G$ in $\RR$ which contains $1$) such that $J(G)= \bigcap_{r \in S} f_r^{-1}[G]$ and so $J(G)$ is $M$-type-definable. We will prove that $J(G)=G$, which shows that $G$ is a left ideal. Then the right version of this argument shows that $G$ is  also a right ideal, so we will be done.

We need to show that $G \subseteq J(G)$.
Let $G(x)$ be the partial type defining $G$, and $J(x)$ --- the partial type defining $J(G)$. Both types are with parameters from $M$.
Take any formula $\varphi(x) \in J(x)$.  It is enough to show that $G \subseteq \varphi(\RR)$.
By compactness, we can find $\psi(x) \in G(x)$ and $s_0,\dots,s_{n-1} \in S$ such that $f_{s_0}^{-1}[\psi(\RR)]  \cap \dots \cap f_{s_{n-1}}^{-1}[\psi(\RR)]] \subseteq \varphi(\RR)$. So we can find $r_0,\dots,r_{n-1} \in R$ such that $f_{r_0}^{-1}[\psi(\RR)] \cap \dots \cap f_{r_{n-1}}^{-1}[\psi(\RR)] \subseteq \varphi(\RR)$.  Since for every $r \in R$, $G \subseteq f_r^{-1}[G]$, we have $G \subseteq f_{r_0}^{-1}[\psi(\RR)] \cap \dots \cap f_{r_{n-1}}^{-1}[\psi(\RR)] \subseteq \varphi(\RR)$. 
\end{proof}

\begin{definition}
$\RRiT_{\topo}: = \mu + \RRiT_M$.
\end{definition}

Since $\RRiT_M$ is the smallest $M$-type-definable two-sided ideal [ring] of bounded index, we get the following corollary.

\begin{corollary}\label{corollary: charact. of R^00_topo}
$\RRiT_{\topo}$ is the smallest $M$-type-definable, bounded index two-sided ideal containing $\mu$ and also the smallest $M$-type-definable, bounded index ring containing $\mu$.
\end{corollary}

\begin{proposition}
The quotient map $\pi \colon R \to \RR/\RRiT_{\topo}$ is the Bohr compactification of the topological ring $R$.
\end{proposition}

\begin{proof}
The proof is a straightforward  adaptation of the proof of \cite[Fact 2.4(ii)]{krupil}, so we will skip it. Let us only remark that, using the notation from the proof of \cite[Fact 2.4(ii)]{krupil}, since  $\ker(f^*)$ is a bounded index two-sided ideal which is an intersection of some sets of the form $\bar{U}$ for $U$ open in $R$, we see that it is a $0$-type-definable, bounded index two-sided ideal containing $\mu$, and so $\RRiT_{\topo} \subseteq \ker(f^*)$ by Corollary \ref{corollary: charact. of R^00_topo}.
\end{proof}


	\section{Classical and definable Bohr compactifications of some matrix groups}
	
	
	Our aim in this section is to describe the definable (in particular classical, taking the full structure) Bohr compactifications of some classical discrete groups. 
We focus on the groups $\UT_n(R)$ and $\T_n(R)$ of (respectively) upper unitriangular and invertible upper triangular matrices over a (unital) ring $R$ and describe their type-definable connected components in order to compute their definable Bohr compactifications. This is done in Subsection \ref{sec:unig}. In Subsection \ref{subsection: special rings}, we apply these general considerations to some classical rings $R$ (such as $\Z$ or $K[\bar X]$), computing explicitly 
the definable Bohr compactifications of $\UT_n(R)$ and $\T_n(R)$ for those rings. In the last subsection, we apply our approach to the topological groups $\UT_n(R)$ and $\T_n(R)$ for $R$ being a topological (unital) ring, computing their classical Bohr compactifications.

In this section, we often write matrices where some of the coefficients are replaced with sets of coefficients to denote the set of matrices in which the coefficients can be (independently) chosen from the sets that replace them. Similarly, we replace submatrices with sets of submatrices.
	
	\subsection{Some linear algebra over rings}\label{subsection: linear algebra}
	
	

	First, we analyze the structure of the group $\UT_n(R)$ for a unital ring $R$. The following belongs to standard linear algebra. A matrix $B \in \UT_{n+1}(R)$ can be written as $\QM{A}{v}{0}{1}$
	for some $A \in \UT_{n}(R)$ and $v \in R^n$. The map \mbox{$\psi \colon \UT_{n+1}(R) \to \UT_{n}(R)$} given by sending $B$ to its upper-left $n \times n$ submatrix $A$ is a group epimorphism. Its kernel consists of all matrices of the form $\QM{I}{v}{0}{1}$, $v \in R^n$, and is naturally isomorphic to $(R,+)^n$. The short exact sequence
	\[1 \to (R,+)^n \to \UT_{n+1}(R) \xrightarrow{\psi} \UT_{n}(R) \to 1\]
	splits via the map $s \colon \UT_{n}(R) \to \UT_{n+1}(R)$ which sends $A$ to $\QM{A}{0}{0}{1}$. Hence, $\UT_{n+1}(R)$ becomes a semidirect product $\UT_{n}(R) \ltimes_\phi (R,+)^n$. With a direct calculation, we verify that the action $\phi\colon \UT_{n}(R) \to \mathrm{Aut}((R,+)^n)$ is just the standard action of $\UT_{n}(R)$ on the $R$-module $R^n$:
	\begin{align*}
	\QM{A}{0}{0}{1}
	\QM{I}{v}{0}{1}
	\QM{A}{0}{0}{1}^{-1}  =  \QM{A}{Av}{0}{1}
	\QM{A^{-1}}{0}{0}{1} = 	\QM{I}{Av}{0}{1}.
	\end{align*}
	Thus, the group operation in \mbox{$\UT_{n}(R) \ltimes_\phi (R,+)^n$} is just \mbox{$(A,v)\cdot(A',v') = (AA', v + Av')$}.

	
	We now perform a similar analysis for $\T_n(R)$. First, consider the following variant of the semidirect product of groups. Suppose that $K, H$ and $N$ are groups, and that there are: a left action $\phi_1 \colon K \to \Aut(N)$ and a right action $\phi_2 \colon H \to \Aut(N)$. For $k \in K, h \in H, n \in N$, write $kn$ and $nh$ in place of $\phi_1(k)(n)$ and $\phi_2(h)(n)$, respectively. The set $K \times H \times N$ can be equipped with the following operation:
	\[(k, h, n) \cdot (k', h', n') = \left(k k', h h', ( k n')(n h')\right).\]
	It is easy to see that this is a group operation if and only if both actions commute, that is if $k(nh) = (kn)h$ for all $k \in K, h \in H, n \in N$. In that case, we will denote such a group as $(K, H) \ltimes_{\phi_1}^{\phi_2} N$. The groups $K \times H$ and $N$ are naturally embedded in $(K, H) \ltimes_{\phi_1}^{\phi_2} N$ as $K \times H \times \{1\}$, and $\{1\} \times \{1\} \times N$ respectively. The subgroup $N \leq (K, H) \ltimes_{\phi_1}^{\phi_2} N$ is normal. The action of $K \times H$ on $N$ by conjugation is as follows:
	\[(k, h, 1) \cdot (1, 1, n) \cdot (k, h, 1)^{-1} = (k, h, 1) \cdot (1, 1, n) \cdot (k^{-1}, h^{-1}, 1) = (1,1,knh^{-1}).\]
	
	Note that if either of the actions $\phi_1, \phi_2$ is trivial, then $(K, H) \ltimes_{\phi_1}^{\phi_2} N$ is just a semidirect product of $K \times H$ and $N$.
	
	Now, consider a matrix $B \in \T_{n+1}(R)$. It can be written as $\QM{A}{v}{0}{r}$
	for some $A \in \T_{n}(R), v \in R^n$, and $r \in R^*$. We consider a product of two matrices represented this way:
	\begin{align*}
	\QM{A}{v}{0}{r}
	\QM{A'}{v'}{0}{r'}  =  \QM{AA'}{Av' + vr'}{0}{rr'}.
	\end{align*}
	From the calculation above, it follows that $\T_{n+1}(R)$ is isomorphic to the group $\left(\T_n(R), R^*\right) \ltimes_{\phi_1}^{\phi_2} (R^n, +)$ with $A \in \T_n(R)$ acting on $R^n$ by $v \mapsto Av$, and $R^*$ acting on $R^n$ by $v \mapsto vr$. Hence, the conjugate of $v \in R^n$ by $(A, r) \in \T_n(R) \times R^*$ is $(Av)r^{-1} = A(vr^{-1})$.

\subsection{Discrete triangular groups}\label{sec:unig}

Recall that $R$ is a unital ring, and $A\subset\RR$ is a small set of parameters.

Our first goal is to describe $\UT(\RR)^{00}_A$, the $A$-type-definable connected component of $\UT_n(\RR)$, along with the quotient $\UT(\RR)/\UT(\RR)^{00}_A$.  
In particular, for $A :=R$, we get a description of the definable Bohr compactification of $\UT_n(R)$; working in  $\Lsetp{R}$, this compactification coincides with the classical Bohr compactification of the discrete group $\UT_n(R)$.


A natural candidate for the component is $\UT_n(\RRiT_A)$.
However, we will see that in general it may happen that $\UT(\RR)^{00}_A \lneq \UT_n(\RRiT_A)$.

Define a sequence $I_{i,A}(\RR)$, $i\in\N_{> 0}$, of $A$-type-definable subgroups of $\RRa$ as follows:

	\begin{itemize}
	    \item $I_{1,A}(\RR) = \RRaT_A$,
	    \item for $i > 0$ let $I_{i+1,A}(\RR)$ be the smallest $A$-type-definable subgroup of $\RRa$ containing the set $\RR \cdot I_{i,A}(\RR)$.
	\end{itemize}
	We clearly have
	\[\RRaT_A = I_{1,A}(\RR) \leq I_{2,A}(\RR) \leq \ldots \leq I_{i,A}(\RR) \leq \ldots \leq \RRiT_A.\]

	Moreover, if for some $i \in \N_{>0}$ the group $I_{i,A}(\RR)$ is a two-sided ideal (or just left ideal) of $\RR$, then $I_{j,A}(\RR) = \RRiT_A$ for all $j \geq i$. Conversely, if $I_{j,A}(\RR)$ is constant for $j \geq i$, then $I_{i,A}(\RR) =  \RRiT_A$ is an ideal. Indeed, since $I_{i,A}(\RR) = I_{i+1,A}(\RR)$, it is a bounded index, $A$-type-definable left ideal contained in $\RRiT_A$, and so it coincides  with $\RRiT_A$ by Corollary \ref{cor:id}.
	
	When $\RR$ and $A$ are fixed, we will omit the parameters and write $I_i$ to denote $I_{i,A}(\RR)$. 

	\begin{proposition}
		\label{prop:utn00}
		\[\UT_n(\RR)^{00}_A  =
		\begin{pmatrix}
		1 & I_1 & I_2 & \ldots & I_{n-2} & I_{n-1} \\
		0 & 1 & I_1 & \ldots & I_{n-3} & I_{n-2} \\
		0 & 0 & 1 & \ldots & I_{n-4} & I_{n-3} \\
		\vdots & \vdots & \vdots & \ddots & \vdots & \vdots \\
		0 & 0 & 0 & \ldots & 1 & I_1 \\
		0 & 0 & 0 & \ldots & 0 & 1 \\
		\end{pmatrix}.\]
	\end{proposition}
	
	While the groups $I_i$ need not be (two-sided) ideals, if $i, j < k$, then for any coset $a + I_i \in \RRa/I_i$ and $b + I_j \in \RRa/I_j$ we have $(a + I_i)(b + I_j) \subseteq ab + I_k$; that is, the coset $ab + I_k$ is well-defined. Consequently, if $S = \sum_s v_s w_s$ where each $w_s$ and each $v_s$ is a coset of $I_{i_s}$ and $I_{j_s}$, respectively, then $S$ can be unambiguously considered as an element of $\RRa/I_{k}$ for any $k$ such that $i_s, j_s < k$ for all $s$. In the result below, the group operation on the set of matrices is defined using this identification.

	\begin{proposition}
		\label{prop:utn00quot}
 The definable Bohr compactification of the (discrete) group $\UT_n(R)$ is
		\[\UT_n(\RR)/\UT_n(\RR)^{00}_R  \cong
		\begin{pmatrix}
		1 & B/I_1 & B/I_2 & \ldots & B/I_{n-2} & B/I_{n-1} \\
		0 & 1 & B/I_1 & \ldots & B/I_{n-3} & B/I_{n-2} \\
		0 & 0 & 1 & \ldots & B/I_{n-4} & B/I_{n-3} \\
		\vdots & \vdots & \vdots & \ddots & \vdots & \vdots \\
		0 & 0 & 0 & \ldots & 1 & B/I_1 \\
		0 & 0 & 0 & \ldots & 0 & 1 \\
		\end{pmatrix},\]
		where $B := \RRa$ and $\cong$ is a topological group isomorphism, with the right hand side equipped with the product topology induced from the logic topologies on the quotients $B/I_i$. 
The quotient $B/I_1$ is exactly the definable Bohr compactification of $(R,+)$.

More precisely, the definable Bohr compactification of $\UT_n(R)$ is the homomomorphism from $\UT_n(R)$ to the above group of matrices given coordinatewise as the quotients by the appropriate $I_i$'s.
	\end{proposition}



	To state the analogous results for the group $\T_n(\RR)$, we need to define another non-decreasing sequence $I_{i,A}'(\RR)$, $i\in\N_{> 0}$, of $A$-type-definable subgroups of $\RRa$ as follows:
	\begin{itemize}
		\item $I_{1,A}'(\RR)$ is the smallest $A$-type-definable subgroup of $\RRa$ which contains $\RRaT_A$ and which is closed under multiplication by $\RR^*$ from both left and right.
		\item for $i > 0$ let $I_{i+1,A}'(\RR)$ be the smallest $A$-type-definable subgroup of $\RRa$ that contains the set $\RR \cdot I_{i,A}'(\RR) \cdot \RR^*$ and that is closed under multiplication by $\RR^*$ from both left and right.
	\end{itemize}
	By definition and induction, we have $ I_{i,A}(\RR) \subseteq I_{i,A}'(\RR) \subseteq  \RRiT_A$ for all $\RR, A, i$. 
Hence, if $I_{j,A}(\RR)$ is constant for $j \geq i$, then $I_{i,A}'(\RR) =  \RRiT_A$.
Also, as before, if $I_{j,A}'(\RR)$ is constant for $j \geq i$, then $I_{i,A}'(\RR) =  \RRiT_A$.

Again, when $\RR$ and $A$ are fixed, we write $I_i'$ to denote $I_{i,A}'(\RR)$. 
	
	\begin{proposition}
		\label{prop:tn00}
		\[\T_n(\RR)^{00}_A  =
		\begin{pmatrix}
		(\RR^*,\cdot)_A^{00} & I_1' & I_2' & \ldots & I_{n-2}' & I_{n-1}' \\
		0 & (\RR^*,\cdot)_A^{00} & I_1' & \ldots & I_{n-3}' & I_{n-2}' \\
		0 & 0 & (\RR^*,\cdot)_A^{00} & \ldots & I_{n-4}' & I_{n-3}' \\
		\vdots & \vdots & \vdots & \ddots & \vdots & \vdots \\
		0 & 0 & 0 & \ldots & (\RR^*,\cdot)_A^{00} & I_1' \\
		0 & 0 & 0 & \ldots & 0 & (\RR^*,\cdot)_A^{00} \\
		\end{pmatrix}.\]
	\end{proposition}

	The group operation in the result below uses the identifications analogous to those discussed before Proposition \ref{prop:utn00quot}:
	
	\begin{proposition}
		\label{prop:tn00quot}
		The definable Bohr compactification of the (discrete) group $\T_n(R)$ is
		\[\T_n(\RR)/\T_n(\RR)^{00}_R  \cong
		\begin{pmatrix}
		P & B/I_1' & B/I_2' & \ldots & B/I_{n-2}' & B/I_{n-1}' \\
		0 & P & B/I_1' & \ldots & B/I_{n-3}' & B/I_{n-2}' \\
		0 & 0 & P & \ldots & B/I_{n-4}' & B/I_{n-3}' \\
		\vdots & \vdots & \vdots & \ddots & \vdots & \vdots \\
		0 & 0 & 0 & \ldots & P & B/I_1' \\
		0 & 0 & 0 & \ldots & 0 & P \\
		\end{pmatrix},\]
		where $P := (\RR^*,\cdot)/(\RR^*,\cdot)_R^{00}$ is the definable Bohr compactification of $(R^*,\cdot)$, $B := \RRa$, and $\cong$ is a topological group isomorphism, with the right hand side equipped with the product topology induced from the logic topologies on the quotients $B/I_i'$.

More precisely, the definable Bohr compactification of $\T_n(R)$ is the homomomorphism from $\T_n(R)$ to the above group of matrices given coordinatewise as the quotients by $(\RR^*,\cdot)_R^{00}$ or by the appropriate $I_i'$'s.
	\end{proposition}	
	We will prove Propositions \ref{prop:utn00}-\ref{prop:tn00quot} later in this subsection.

From now on, when we compute Bohr compactications, we will be describing them only as compact groups, skipping the information about the actual homomorphisms from the groups in question to these compact groups, since these homomorphisms are always as in the last parts of Propositions \ref{prop:utn00quot} and \ref{prop:tn00quot}.


	The descriptions of the definable Bohr compactifications of $\UT_n(R)$ and $\T_n(R)$ given by Propositions \ref{prop:utn00quot} and \ref{prop:tn00quot} can be significantly improved under the following condition on the ring $R$:
	\[I_{i,R}(\RR) = \RRiT_R \text{ for all } i \geq 2.\eqno{(\dagger)}\]
	The condition asserts exactly that the sequence $I_{i,R}(\RR)$ stabilizes after (at most) two steps. Assuming $(\dagger)$, each quotient $\RRa/I_{i,R}(\RR)$ and  $\RRa/I_{i,R}'(\RR)$ for $i \geq 2$ is the ring definable Bohr compactification $\dBohr{R}$ of $R$. Hence, by Propositions \ref{prop:utn00quot} and \ref{prop:tn00quot}, we get:
	
	\begin{corollary}
		\label{cor:utn00quot with dagger}
		Assume $R$ satisfies $(\dagger)$. Then the definable Bohr compactification of the group $\UT_n(R)$ is
		\[\begin{pmatrix}
		1 & \dBohr{(R, +)} & \dBohr{R} & \ldots & \dBohr{R} & \dBohr{R} \\
		0 & 1 & \dBohr{(R, +)} & \ldots & \dBohr{R} & \dBohr{R} \\
		0 & 0 & 1 & \ldots & \dBohr{R} & \dBohr{R} \\
		\vdots & \vdots & \vdots & \ddots & \vdots & \vdots \\
		0 & 0 & 0 & \ldots & 1 & \dBohr{(R, +)} \\
		0 & 0 & 0 & \ldots & 0 & 1 \\
		\end{pmatrix}\]
		considered with the product topology. Also, the definable Bohr compactification of the group $\T_n(R)$ is
\[\begin{pmatrix}
		\dBohr{(R^*,\cdot)} & (\RR,+)/I_1' & \dBohr{R} & \ldots & \dBohr{R} & \dBohr{R} \\
		0 & \dBohr{(R^*,\cdot)} & (\RR,+)/I_1' & \ldots & \dBohr{R} & \dBohr{R} \\
		0 & 0 & \dBohr{(R^*,\cdot)} & \ldots & \dBohr{R} & \dBohr{R} \\
		\vdots & \vdots & \vdots & \ddots & \vdots & \vdots \\
		0 & 0 & 0 & \ldots & \dBohr{(R^*,\cdot)} & (\RR,+)/I_1' \\
		0 & 0 & 0 & \ldots & 0 & \dBohr{(R^*,\cdot)} \\
		\end{pmatrix}\]
considered with the product topology.
	\end{corollary}
	
	In the next subsection, we will consider several classes of rings, each time showing that they satisfy $(\dagger)$. This motivates the following:
	\begin{question}\label{quest: finitely many steps in the sequence I_i}
		Does $(\dagger)$ hold for every ring $R$?
	\end{question}

Condition $(\dagger)$ is strongly related to Question \ref{quest: 000=00}, which is explained in the next two lemmas.

\begin{lemma}\label{lem: occurrence of 000}
The subgroup $J$ of $(\bar R, +)$ generated by $\RR \cdot \RRaT_R$ is precisely $\RRiI_R$.
\end{lemma}
\begin{proof}
As $\RRa$ is abelian, $\RRaT_R = \RRaI_R \subseteq \RRiI_R$ by \cite[Theorem 0.5]{krupil}. Hence, $J$ is contained in $\RRiI_R$. On the other hand, $J$ is an $R$-invariant left ideal which contains $\RRaT_R$ and so has bounded index, hence it must contain $\RRiI_R$ by Corollary \ref{cor:id}.
\end{proof}
	
\begin{lemma}\label{lem: conditions equivalent to 000=00}
Let $J$ be the subgroup of  $(\bar R, +)$ generated by $\RR \cdot \RRaT_R$. The following conditions are equivalent.
		\begin{enumerate}[(i)]
			\item $J$ is type-definable.
\item $J$ is generated by $\RR \cdot \RRaT_R$ in finitely many steps.
\item  $\RRiI_R = \RRiT_R$.
		\end{enumerate}
	If the above equivalent conditions hold, then $(\dagger)$ holds for $R$.
	\end{lemma}
	\begin{proof}
The implication (i) $\rightarrow$ (ii) follows from Theorem 3.1 of \cite{New}; and (ii) $\rightarrow$ (i) is trivial. The equivalence (i) $\leftrightarrow$ (iii) follows from Lemma \ref{lem: occurrence of 000}.
	\end{proof}

	A positive answer to Question \ref{quest: 000=00} is the assertion that condition (iii) of Lemma \ref{lem: conditions equivalent to 000=00} holds, yielding $(\dagger)$ and the descriptions of the definable Bohr compactifications of $\UT_n(R)$ and $\T_n(R)$ given by Corollary \ref{cor:utn00quot with dagger}.
%
So Question \ref{quest: 000=00} can be restated in the following enriched form.

\begin{question}\label{question: to be answered in the next paper}
Do the equivalent conditions from Lemma \ref{lem: conditions equivalent to 000=00} hold for every unital ring? If yes, is there a bound on the number of steps which are needed to generate a group by $\RR \cdot \RRaT_R$ which works for all rings $R$?
\end{question}

We expect a positive answer to this question (so also to Question \ref{quest: finitely many steps in the sequence I_i}). This will be dealt with in a forthcoming paper of the third author and Tomasz Rzepecki. In the next subsection, we will give a positive answer in several concrete examples.

Let us only argue here that in order to answer Questions \ref{quest: 000=00} and \ref{quest: finitely many steps in the sequence I_i} for commutative, unital rings $R$ in the full language $\Lsetp{R}$, we can restrict the context to polynomial rings over $\Z$ in possibly infinitely many variables. This essentially follows from the fact that for each commutative, unital ring there is a ring of polynomials $\Z[\bar X]$ (where $\bar X$ is a tuple of possibly infinitely many variables) and an epimorphism $f \colon \Z[\bar X] \to R$. Indeed, let us work in the two-sorted structure $M$ with sorts $R$ and $\Z[\bar X]$ in the language $\Lsetp{M}$. Then all the relevant ``components'' associated with $R$ computed in $\Lsetp{R}$ coincide with the ones computed in $\Lsetp{M}$, and similarly for the ring $\Z[\bar X]$; hence, we can work in $\Lsetp{M}$. Put $P:=\Z[\bar X]$, and let $\bar P$ be the interpretation of $P$ in the monster model $\bar M$. Finally, since $f$ is a 0-definable ring epimorphism, by Remark \ref{remark: f[00]=00}, we have $f[\bar P^{000}_\emptyset] = \RRiI_\emptyset$,  $f[\bar P^{00}_\emptyset] = \RRiT_\emptyset$, and we easily check that $f[I_{i,\emptyset}(\bar P)]= I_{i,\emptyset}(\bar R)$ for all $i$. 

The same holds for non-commutative rings, using free rings in non-commuting variables in place of polynomial rings.\\




We now show a number of lemmas needed in the proofs of Propositions \ref{prop:utn00}-\ref{prop:tn00quot}.
We will be using notations and observations from Subsection \ref{subsection: linear algebra}.


	\begin{lemma}
		\label{lem:twosemi}
		Let $K$, $H$, and $N$ be $0$-definable groups and $G: = (K, H) \ltimes_{\phi_1}^{\phi_2} N$ with $0$-definable actions $\phi_1, \phi_2$. Then $\G^{00}_A = (\bar{K}^{00}_A \times \bar{H}^{00}_A) \ltimes_{\phi_1}^{\phi_2} N'$, where $N'$ is the smallest $A$-type-definable, bounded index subgroup of $\bar{N}$ invariant under the actions of both $\bar{K}$ and $\bar{H}$ on $\bar{N}$.
	\end{lemma}
	\begin{proof}
		First observe that a subgroup $N_0 \leq \bar{N}$ is invariant under the actions of $\bar{K}$ and $\bar{H}$ if and only if it is invariant under conjugation by elements of $\bar{K} \times \bar{H}$. 
		The group $\G^{00}_A \cap \bar{N}$ is a bounded index, $A$-type-definable subgroup of $ \bar{N}$ invariant under the action of $\bar{K} \times \bar{H}$ by conjugation, so it contains $N'$. The group $\G^{00}_A \cap (\bar{K} \times \bar{H})$ is an $A$-type-definable subgroup of $\bar{K} \times \bar{H}$ of bounded index, so it contains $(\bar{K} \times \bar{H})^{00}_A = \bar{K}^{00}_A \times \bar{H}^{00}_A$. Thus, $(\bar{K}^{00}_A \times \bar{H}^{00}_A) \ltimes_{\phi_1}^{\phi_2} N' \subseteq \G^{00}_A$. Since $N'$ is invariant under both group actions, $(\bar{K}^{00}_A \times \bar{H}^{00}_A) \ltimes_{\phi_1}^{\phi_2} N'$ is a group. It is $A$-type-definable and with bounded index, so we get $\G^{00}_A = (\bar{K}^{00}_A \times \bar{H}^{00}_A) \ltimes_{\phi_1}^{\phi_2} N'$.
	\end{proof}
	
	\begin{corollary}
		\label{cor:semi}
		Let $H$ and $N$ be $0$-definable groups and $G := H \ltimes_\phi N$ with a $0$-definable action $\phi$. Then $\G^{00}_A = \bar{H}^{00}_A \ltimes_{\phi} N'$, where $N'$ is the smallest $A$-type-definable, bounded index subgroup of $\bar{N}$ invariant under the action of $\bar{H}$ on $\bar{N}$.
	\end{corollary}

	\begin{lemma}
		\phantomsection\label{lem:utninv}
		\begin{enumerate}[(i)]
			\item Let $G :=  \UT_{n}(R) \ltimes_\phi (R,+)^n$, where $\phi(A)(v) := Av$. Then the smallest bounded index, invariant under the action of $\UT_n(\RR)$, and $A$-type-definable subgroup $N'$ of $(\RR,+)^n$ is equal to \[I_n \times I_{n-1} \times \ldots \times I_1.\]
			\item Let $G := \left(\T_{n}(R) \times (R^*,\cdot)\right) \ltimes_{\phi_1}^{\phi_2} (R,+)^n$, where $\phi_1(A)(v) := Av$ and $\phi_2(r)(v) := vr$. Then the smallest bounded index, invariant under the actions of both $\T_n(\RR)$ and $\RR^*$, and $A$-type-definable subgroup $N'$ of $(\RR,+)^n$ is equal to \[I_n' \times I_{n-1}' \times \ldots \times I_1'.\]
		\end{enumerate}
	\end{lemma}
	\begin{proof}
For $k \leq n$, let $S_k$ be the image of the embedding of $\RRa^k$ into $(\RR,+)^n$ by the map $(v_k,\ldots,v_1) \mapsto (0,\ldots,0,v_k,\ldots,v_1)$. We freely identify $S_k$ with $\RRa^k$.

First we show (i). We prove the following by induction on $k$:
		\[N' \cap S_k \supseteq I_k \times \ldots \times I_1.\]
		For $k=1$, simply observe that $N' \cap S_1$ is an $A$-type-definable subgroup of $\RRa$ of bounded index so it must contain $I_1$. Now, suppose the statement holds for some $k > 0$. Let
		\[C = \{c \in \RR: (c,0,0,\ldots,0) \in N' \cap S_{k+1}\}.\]
		We have $N' \cap S_{k+1} \supseteq C \times I_k \times \ldots \times I_1$. We conclude the induction by showing that $C \supseteq I_{k+1}$.
		Let $x \in I_k$ be arbitrary and take $v := (0,x,0,\ldots,0) \in N' \cap S_{k+1}$. For any $r \in \RR$ there is an $A \in \UT_n(\RR)$ such that $\phi(A)(v) = (rx,x,0,\ldots,0) \in S_{k+1}$. As $N'$ is invariant under $\UT_n(\RR)$, we have $\phi(A)(v) \in N'$ and also \mbox{$\phi(A)(v) - v = (rx,0,0,\ldots,0) \in N' \cap S_{k+1}$}. This shows that $C$ contains the set $\RR \cdot I_k$. 
Therefore, since $C$ is an $A$-type-definable subgroup of $\RRa$ of bounded index, it contains $I_{k+1}$.

		We have that $N'=  N' \cap S_n \supseteq I_n \times \ldots \times I_1$. As $I_n \times \ldots \times I_1$ is $A$-type-definable with bounded index, it remains to show that it is invariant under the action of $\UT_n(\RR)$. Take $v = (v_n, v_{n-1}, \ldots, v_1) \in \RR^n$. For a unitriangular matrix $A$, $Av$ is of the form $(v_n + v_n', v_{n-1} + v_{n-1}', \ldots, v_2 + v_2', v_1)$, where each $v_i'$ is an $\RR$-linear combination of $\{v_j : j < i\}$. So $v \in I_n \times \ldots \times I_1$ implies $Av \in I_n \times \ldots \times I_1$, since $I_i \supseteq I_i + \RR I_{i-1} + \ldots + \RR I_1$.

		We now prove (ii). First, let $r, r' \in \RR^*$ and let $I$ denote the $n \times n$ identity matrix. Since $N'$ is closed under the actions $\phi_1, \phi_2$, we have $rN'r'=(rI)N'r' \subseteq N'$, so $N'$ is closed under multiplication by $\RR^*$ from both left and right.
		
		Now, similarly to (i), we prove by induction on $k$ that \[N' \cap S_k \supseteq I_k' \times \ldots \times I_1'.\]
		For $k=1$, we again observe that $N' \cap S_1$ is an $A$-type-definable subgroup of $\RRa$ of bounded index. It is closed under multiplication by $\RR^*$ from both sides, so it must contain $I_1'$. Now, suppose the statement holds for some $k > 0$. Define $C$ as in the proof of item (i). Then $N' \cap S_{k+1} \supseteq C \times I_k' \times \ldots \times I_1'$ and we need to show $C \supseteq I_{k+1}'$.


		Let $x \in I_k'$ be arbitrary and take $v := (0,x,0,\ldots,0) \in N' \cap S_{k+1}$. For any $r \in \RR$ and $r' \in \RR^*$ there is an $A \in \T_n(\RR)$ such that $Avr' = (rxr',xr',0,\ldots,0) \in S_{k+1}$. We have $Avr' \in N'$ and also \mbox{$Avr' - vr' = (rxr',0,0,\ldots,0) \in N' \cap S_{k+1}$}. This shows that $C$ contains the set $\RR \cdot I_k' \cdot \RR^*$. Since $C$ is an $A$-type-definable subgroup of $\RRa$ of bounded index, closed under multiplication by $\RR^*$ from both left and right, it contains $I_{k+1}'$.

		We now have that $N'=  N' \cap S_n \supseteq I_n' \times \ldots \times I_1'$. As $I_n' \times \ldots \times I_1'$ is $A$-type-definable with bounded index, and clearly invariant under multiplication by $\RR^*$ from the right, it remains to show that it is invariant under the action of $\T_n(\RR)$. Take $v = (v_n, v_{n-1}, \ldots, v_1) \in \RR^n$. For a triangular matrix $A$, $Av$ is of the form $(r_n v_n + v_n', r_{n-1}v_{n-1} + v_{n-1}', \ldots, r_2 v_2 + v_2', r_1 v_1)$, where for each $i$, $v_i'$ is an $\RR$-linear combination of $\{v_j : j < i\}$ and $r_i \in \RR^*$. So $v \in I_n' \times \ldots \times I_1'$ implies $Av \in I_n' \times \ldots \times I_1'$, since $I_i' \supseteq \RR^* I_i' + \RR I_{i-1}' + \ldots + \RR I_1'$.
	\end{proof}

	We are now ready to prove the previously stated results.
	\begin{proof}[Proof of Proposition \ref{prop:utn00}]
		By Corollary \ref{cor:semi} and Lemma \ref{lem:utninv}(i), we have
		\begin{align*}
		\UT_{n+1}(\RR)^{00}_A & \cong \UT_{n}(\RR)^{00}_A \ltimes (I_n \times \ldots \times I_1)\\
		& \cong \QM{\UT_{n}(\RR)^{00}_A}{\begin{array}{c}I_n \\ \vdots \\ I_1 \end{array}}{0}{1},
		\end{align*}
where the isomorphisms are the obvious ones so that the first and the last group are in fact equal. Hence, the result follows by induction on $n$.
	\end{proof}
	
	\begin{proof}[Proof of Proposition \ref{prop:utn00quot}]
Write $H$ for the group of matrices on the right hand side of the formula in Proposition \ref{prop:utn00quot}.
		Let $F \colon \UT_n(\RR) \to H$ be the map sending a matrix $[a_{ij}] \in \UT_n(\RR)$ to the matrix $[b_{ij}] \in H$ defined by
		\[
		b_{ij} :=
		  \begin{cases} 
		  a_{ij} + I_{j-i} & \text{if } i < j, \\
  		  1 & \text{if } i = j, \\
  		  0 & \text{otherwise.}
		  \end{cases}
		\]
		As the group operation in $H$ is the ordinary matrix multiplication after the identification of the product of cosets $a + I_i$, $b + I_j$ with $ab + I_k$ for all $a,b \in \RR$ and $i,j < k$, the map $F$ is a group homomorphism. 
It is clearly onto, and Proposition \ref{prop:utn00} implies that $\ker (F) = \UT_n(\RR)^{00}_A$. Hence, $\UT_n(\RR)/\UT_n(\RR)^{00}_A \cong H$ as an abstract group.
By compactness of the logic topologies, in order to see that this isomorphism is a homeomorphism, it is enough to check that it is continuous. But this is clear, as the preimage by $F$ of a subbasic closed set $S$ in the product topology on $H$ (i.e. $S$ consists of all matrices in $H$ whose fixed $(i,j)$-th entry belongs to a fixed closed subset of $B/I_{j-i}$) is type-definable.
	\end{proof}

The proofs of Propositions \ref{prop:tn00}-\ref{prop:tn00quot} are similar to the two previous ones.

	\subsection{Connected components of abelian groups via characters}

In this subsection, we give a description in terms of characters of the type-definable connected component of any abelian group. 
It will be needed to get descriptions of some Bohr compactifications in the next subsection, and may prove to be useful in future studies. 

	Let $G$ be an abelian group definable in a structure $M$. Recall that $\Hom(G, S^1)$ is the group of all homomorphisms from $G$ to the compact group $S^1=\R/\Z=[-\frac{1}{2},\frac{1}{2})$. 
By $\defi{\Hom}(G,S^1)$ we denote the subgroup consisting of all definable homomorphisms in the sense explained before Fact \ref{cor:gpp}. Note that if all subsets of $G$ are definable, then $\defi{\Hom}(G,S^1) = \Hom(G,S^1)$.

The next fact follows from the proofs of Lemma 3.2 and Proposition 3.4 of \cite{GPP}.
    \begin{fact}\label{fact: extension of definable function}
Each $\chi\in\defi{\Hom}(G,S^1)$ extends uniquely to an $M$-definable 
homomorphism $\bar{\chi} \in \Hom(\bar{G},S^1)$, where $M$-definable means that the preimages of all closed subsets of $S^1$ are $M$-type-definable subsets of $\bar G$.
    \end{fact}
Lemma 3.2 of \cite{GPP} provides the following construction of $\bar{\chi}$ from the fact above. Let $g \in \bar{G}$ and let $p(x) := \tp(g/M)$. For a formula $\phi(x) \in p$, let $\cl\left(\chi[\phi(G)]\right)$ denote the closure of $\chi[\phi(G)]$ in $S^1$. The set $\bigcap_{\phi \in p} \cl\left(\chi[\phi(G)]\right)$ is shown to be a singleton in $S^1$, and $\bar{\chi}(g)$ is defined to be the unique element of this singleton.



\begin{proposition} \label{fact:char3}
Suppose that $G$ is an abelian group $0$-definable in $M$.
Then
		\[(\bar{G},+)^{00}_M = \bigcap_{\chi\in\defi{\Hom}(G,S^1)} \bigcap_{m\in\N_{>0}}  \bar{\chi}^{-1}\left[\left(-\frac{1}{m},\frac{1}{m}\right)\right] =  \bigcap_{\chi\in\defi{\Hom}(G,S^1)} \bigcap_{m\in\N_{>0}}  \bar{\chi}^{-1}\left[\left[-\frac{1}{m},\frac{1}{m}\right]\right].\] 
	\end{proposition}
	\begin{proof}
The second equality is obvious. So we focus on the first equality.

 ($\subseteq$) Observe that  for every $\chi\in\defi{\Hom}(G,S^1)$, 
 \[\ker (\bar{\chi}) = \bigcap_{m\in\N_{>0}}  \bar{\chi}^{-1}\left[\left(-\frac{1}{m},\frac{1}{m}\right)\right]=\bigcap_{m\in\N_{>0}}  \bar{\chi}^{-1}\left[\left[-\frac{1}{m},\frac{1}{m}\right]\right]\]
 is an $M$-type-definable subgroup of $\bar{G}$ of bounded index.
 Hence, $\ker (\bar{\chi})$ contains $(\bar{G},+)^{00}_M$. 

($\supseteq$) Take $a\in \bar{G}\setminus (\bar{G},+)^{00}_M$. Let 
\[i\colon \bar{G} \to \bar{G}/ (\bar{G},+)^{00}_M \] 
be the ($M$-definable) quotient map. Since $i(a)$ is not the neutral element and $\bar{G}/ (\bar{G},+)^{00}_M$ is a compact abelian group, the second part of Fact \ref{fact:bohrPontryagin} yields $\varphi\in \Hom_c(\bar{G}/ (\bar{G},+)^{00}_M,S^1)$ with $\varphi(i(a)) \ne 0$. Then $\chi':= \varphi \circ i \colon \bar G \to S^1$ is a character which is definable over $M$. Hence, by Fact \ref{fact: extension of definable function}, $ \chi: = \chi' |_{G} \colon G \to S^1$ is a definable character with $\bar \chi =\chi'$. We get that $\bar{\chi}(a)\neq 0$, so $a\not\in\bar{\chi}^{-1}\left[\left(-\frac{1}{m},\frac{1}{m}\right)\right]$ for any $m\in\N$ such that $\frac{1}{m}<|\bar{\chi}(a)|$. 
	\end{proof}

\begin{remark}\label{remark: open vs closed intervals}
Let $G$ be any group equipped with the full structure, and $\chi \colon G \to S^1$ a (0-definable) character. Let $m >1$. Take the $0$-definable set $D: =\chi^{-1}\left[\left(-\frac{1}{m},\frac{1}{m}\right)\right]$ and write $\bar{D}$ for its interpretation in $\bar{G}$. Then:
\begin{enumerate}
\item[(i)]  $\bar D \subseteq \bar \chi^{-1}\left[\left[ -\frac{1}{m}, \frac{1}{m}\right]\right]$;
\item[(ii)]  $\bar D \supseteq \bar \chi^{-1}\left[\left(-\frac{1}{m},\frac{1}{m}\right)\right]$.
\end{enumerate}
\end{remark}

\begin{proof}
(i) The right hand side of the inclusion is 0-type-definable. If the inclusion fails, then there is a 0-definable subset $P$ of $G$ such that $\bar D \setminus \bar P$ is non-empty and disjoint from $\bar \chi^{-1}\left[\left[ -\frac{1}{m}, \frac{1}{m}\right]\right]$. But then we can find $r \in D \setminus P$. Since $r \in G$ and $r \in \bar D \setminus \bar P$, we have that $\chi(r) =\bar \chi(r)$ is not in $\left[-\frac{1}{m},\frac{1}{m }\right]$, a contradiction with the definition of $D$ and the fact that $r \in D$.

(ii) If this fails, then there is $r$ in $\bar \chi^{-1}\left[\left(-\frac{1}{m},\frac{1}{m}\right)\right] \cap \bar D^c$. Hence, by the definition of $\bar \chi$, we get that $\bar \chi(r)$ is in the closure of $\chi[D^c] \subseteq (-\frac{1}{m},\frac{1}{m})^c$, so $\bar \chi(r) \notin (-\frac{1}{m},\frac{1}{m})$, a contradiction.
\end{proof}

By Proposition  \ref{fact:char3} and Remark \ref{remark: open vs closed intervals}, we get


\begin{corollary}
Let $G$ be an abelian group equipped with the full structure. Then
\[(\bar{G},+)^{00}_\emptyset = \bigcap_{\chi\in\Hom(G,S^1)} \bigcap_{m\in\N_{>0}}  \overline{{\chi}^{-1}\left[\left(-\frac{1}{m},\frac{1}{m}\right)\right]}.\]
\end{corollary}

	\subsection{Triangular groups over some classical rings}\label{subsection: special rings}


	We apply Propositions \ref{prop:utn00quot} and \ref{prop:tn00quot} (more precisely, Corollary \ref{cor:utn00quot with dagger})
	to compute definable (so also classical by equipping the ring of coefficients with the full structure) Bohr compactifications of $\UT_n(R)$ and $\T_n(R)$ for the following classical rings $R$: fields, $\Z$, $K[\bar X]$ or even $K[G]$ (where $K$ is a field and $G$ is a group or semigroup).

	For each of the above classes of rings, we first consider the group $\UT_n(R)$. 
We show that the set $\bar{R} \cdot (\bar{R}, +)^{00}_R$ generates a group in finitely many steps, whence condition (ii) of Lemma \ref{lem: conditions equivalent to 000=00} is satisfied. This shows that $(\dagger)$ holds for each of the considered rings, so we can apply Corollary \ref{cor:utn00quot with dagger} to compute the definable Bohr compactification of $\UT_n(R)$. 
In fact, in these examples, the set $\bar{R} \cdot (\bar{R}, +)^{00}_R$ generates $\RRiT_R$ in one step, i.e.  $\bar{R} \cdot (\bar{R}, +)^{00}_R=\RRiT_R$. (On the other hand, one can show that the case of $\Z[X]$ equipped with the full structure requires exactly two steps, which will be shown in the aforementioned forthcoming paper of the third author with Tomasz Rzepecki).
For each $R$, after dealing with the compactification of $\UT_n(R)$, we follow with the computation of the compactification of $\T_n(R)$.

	We begin with the case of an infinite field $R = K$. For any $A$, $\bar{K} \cdot (\bar{K}, +)^{00}_A = \bar{K}$ and so for all $i \geq 2$ we have $I_{i,A}(\bar{K}) = \bar{K}$, the only non-trivial ideal of $\bar{K}$. Corollary \ref{cor:utn00quot with dagger} gives us 
that the definable Bohr compactification $\dBohr{\UT_n(K)}$ of $\UT_n(K)$ is 
	\[
	\begin{pmatrix}
	1 & \dBohr{(K, +)} & 0 & \ldots & 0 & 0 \\
	0 & 1 & \dBohr{(K, +)} & \ldots & 0 & 0 \\
	0 & 0 & 1 & \ldots & 0 & 0 \\
	\vdots & \vdots & \vdots & \ddots & \vdots & \vdots \\
	0 & 0 & 0 & \ldots & 1 & \dBohr{(K, +)} \\
	0 & 0 & 0 & \ldots & 0 & 1 \\
	\end{pmatrix} \cong \left(\dBohr{(K, +)}\right)^{n-1}.\]
We similarly note that for any $A$ we have $I_{i,A}'(\bar{K}) = \bar{K}$ for all $i \geq 1$, so, by Corollary \ref{cor:utn00quot with dagger}, the definable Bohr compactification $\dBohr{\T_n(K)}$ of $\T_n(K)$ is
	\[
	\begin{pmatrix}
	\dBohr{(K^*, \cdot)} & 0  & \ldots & 0 \\
	0 & \dBohr{(K^*, \cdot)} & \ldots & 0 \\
	\vdots & \vdots & \ddots & \vdots \\
	0 & 0 & \ldots & \dBohr{(K^*, \cdot)} \\
	\end{pmatrix} \cong \left(\dBohr{(K^*, \cdot)}\right)^{n}.\]
	
	We now work with $R := \Z$. 
	
	\begin{lemma}
	\label{lem:UTZ}
		$\ZZ \cdot \ZZaT_\Z=\ZZaD$.
	\end{lemma}
\begin{proof}
Since $(\ZZ,+)^0 =\Z^0$ is an ideal, we clearly have $ \ZZ \cdot \ZZaT_\Z \subseteq \ZZaD$, so it is remains to prove $\ZZaD \subseteq \ZZ \cdot \ZZaT_\Z$. The group $\ZZaT_\Z$ is the intersection of a downward directed by inclusion family $\{P_i(\ZZ)\}_{i\in I}$ of 0-definable sets. For every $i \in I$ we can find $n_i \in P_i(\Z) \setminus \{0\}$. Then $n_i \cdot \ZZ \subseteq P_i(\ZZ) \cdot \ZZ = \ZZ \cdot P_i(\ZZ)$. Thus, $(\ZZ,+)^0 = \bigcap_{n\in\N_{>0}}n\cdot\ZZ \subseteq \ZZ\cdot  P_i(\ZZ)$. By compactness, we conclude that 
\[(\ZZ,+)^0\subseteq \ZZ \cdot \bigcap_{i \in I} P_i(\ZZ) =  \ZZ \cdot \ZZaT_\Z.\qedhere\]
\end{proof}

This lemma implies that $(\dagger)$ holds for $R=\Z$.
The quotient $\ZZa/I_{1,\Z}(\ZZ)$ is the definable Bohr compactification $\dBohr{(\Z,+)}$ of $(\Z,+)$, whereas $\ZZa/I_{2,\Z}(\ZZ) = \ZZa/\ZZaD$ is $\hat{\Z}$, i.e. the profinite completion of $\Z$. 
So, by Corollary \ref{cor:utn00quot with dagger}, we get
that the definable Bohr compactification $\dBohr{\UT_n(\Z)}$ of $\UT_n(\Z)$ is 
			\[\begin{pmatrix}
			1 & \dBohr{(\Z,+)} & \hat{\Z} & \ldots & \hat{\Z} & \hat{\Z} \\
			0 & 1 & \dBohr{(\Z,+)} & \ldots & \hat{\Z} & \hat{\Z} \\
			0 & 0 & 1 & \ldots & \hat{\Z} & \hat{\Z} \\
			\vdots & \vdots & \vdots & \ddots & \vdots & \vdots \\
			0 & 0 & 0 & \ldots & 1 & \dBohr{(\Z,+)} \\
			0 & 0 & 0 & \ldots & 0 & 1 \\
			\end{pmatrix}.\]
	
    One easily gets a description of the (topological) connected component of $\dBohr{\UT_n(\Z)}$. For a topological group $G$, we will denote its topological connected component as $G^{t}$ in order to avoid confusions with its model-theoretic components.
	\begin{corollary}
		${\left(\dBohr{\UT_n(\Z)}\right)}^t \cong \left(\left(\dBohr{(\Z,+)}\right)^t\right)^{n-1}.$
	\end{corollary}
	\begin{proof}
	As the ring $\hat{\Z}$ is totally disconnected, we have
	\begin{align*}
	\left(\dBohr{\UT_n(\Z)}\right)^t & \cong
	\begin{pmatrix}
	1 & \dBohr{(\Z,+)} & 0 & \ldots & 0 & 0 \\
	0 & 1 & \dBohr{(\Z,+)} & \ldots & 0 & 0 \\
	0 & 0 & 1 & \ldots & 0 & 0 \\
	\vdots & \vdots & \vdots & \ddots & \vdots & \vdots \\
	0 & 0 & 0 & \ldots & 1 & \dBohr{(\Z,+)} \\
	0 & 0 & 0 & \ldots & 0 & 1 \\
	\end{pmatrix}^t \\
	& \cong \left(\left(\dBohr{(\Z,+)}\right)^{n-1}\right)^t = \left(\left(\dBohr{(\Z,+)}\right)^t\right)^{n-1}. \qedhere
	\end{align*}
	\end{proof}


	Moving to $\T_n(\Z)$, observe that $\ZZ^* = \Z^* =  \{1, -1\}$, and hence $I_i' = I_i$ for all $i$. Then, by Proposition \ref{prop:tn00quot} or Corollary \ref{cor:utn00quot with dagger}, 
		the definable Bohr compactification $\dBohr{\T_n(\Z)}$ of $\T_n(\Z)$ is
		\[\begin{pmatrix}
		\pm 1 & \dBohr{(\Z,+)} & \hat{\Z} & \ldots & \hat{\Z} & \hat{\Z} \\
		0 & \pm 1 & \dBohr{(\Z,+)} & \ldots & \hat{\Z} & \hat{\Z} \\
		0 & 0 & \pm 1 & \ldots & \hat{\Z} & \hat{\Z} \\
		\vdots & \vdots & \vdots & \ddots & \vdots & \vdots \\
		0 & 0 & 0 & \ldots & \pm 1 & \dBohr{(\Z,+)} \\
		0 & 0 & 0 & \ldots & 0 & \pm 1 \\
		\end{pmatrix}.\]


	We are now interested in the rings of polynomials $R:=K[\bar X]$, where $K$ is an arbitrary infinite field and $\bar X$ is a (possibly infinite) tuple of variables. We show that $\RR \cdot \RRaT_R =\bar R =\RRiD_R$ (i.e. again $\RR \cdot \RRaT_R$ generates $\RRiD_R$ in a single step). 
In fact, we will work more generally with any ring $R$ containing an infinite subfield $K$, covering also rings of the form $K[G]$, where $G$ is a group or semigroup.

Recall the notion of a thick set from \cite[Definition 3.1]{modcon}. 

\begin{definition}
A subset $D$ of a group is said to be thick if it is symmetric and there is a natural number $n>0$ such that for any elements $g_0,\dots,g_{n-1}$ there exist $i<j<n$ with $g_i^{-1}g_j \in D$.
\end{definition}
By compactness, it is clear that for any $A$-definable group $G$, each definable superset $\bar D$ of $\bar G^{00}_A$ contains a definable superset of $\bar G^{00}_A$ that is thick in $\bar G$, namely $\bar D \cap \bar D^{-1}$. Hence $\bar G^{00}_A$ is the intersection of some directed family of $A$-definable thick subsets of $\bar G$.
	
	Note that for an arbitrary (unital) ring $R$, $\RR \cdot \RRaT_A \subseteq \RRiT_A = \RRiD_A$. Hence, by compactness, we get
\begin{lemma}\label{very basic lemma2}
Let $R$ be any ring. 
\begin{enumerate}
\item[(i)] $\RR \cdot \RRaT_A =\RRiD_A$ if and only if for every $A$-definable superset $P$ of $(\RR,+)^{00}_A$ there is an $A$-definable two-sided (or just left or right) ideal $P'$ of $\RR$ of finite index with $ P' \subseteq \bar R P$.
\item[(ii)] $\RR \cdot \RRaT_A =\bar R$  if and only if for every $P$ as in (i), $\bar R = \bar R P$.
\item[(iii)] $\RR^* \cdot \RRaT_A =\bar R$  if and only if for every $P$ as in (i), $\bar R = \RR^* P$.
\end{enumerate}
\end{lemma}

	\begin{proposition}\label{proposition: one step}
		Let $R$ be any ring containing an infinite field $K$ (e.g. $R=K[\bar X]$), equipped with any structure. Then $\RR \cdot \RRaT_A =\bar R =\RRiD_A$.
	\end{proposition}
	\begin{proof}
		We need to check that the right hand side of item (ii) from Lemma \ref{very basic lemma2} is satisfied. For this, take any $A$-definable symmetric subset $P$ of $\bar R$ containing $(\RR,+)^{00}_A$. Then $P$ is thick. So $P \cap K$ is thick in $K$, hence there is a non-zero $d \in P \cap K$ (as $K$ is infinite). Since $d$ is invertible, $\bar R d =\bar R$. Thus, we have proved that $\bar R P =\bar R$, so we are done. 
	\end{proof}

	As a corollary, we extend the description of the definable Bohr compactification of $\UT_n(K)$ from the beginning of the subsection to $\UT_n(R)$ for any $R$ containing an infinite field $K$. By the last proposition, $I_{i,R}(\RR) = \RR$ for all $i > 1$, so Corollary \ref{cor:utn00quot with dagger} yields the following description of $\dBohr{\UT_n(R)}$:
	\[
	\begin{pmatrix}
	1 & \dBohr{(R,+)} & 0 & \ldots & 0 & 0 \\
	0 & 1 & \dBohr{(R,+)} & \ldots & 0 & 0 \\
	0 & 0 & 1 & \ldots & 0 & 0 \\
	\vdots & \vdots & \vdots & \ddots & \vdots & \vdots \\
	0 & 0 & 0 & \ldots & 1 & \dBohr{(R,+)} \\
	0 & 0 & 0 & \ldots & 0 & 1 \\
	\end{pmatrix} \cong \left(\dBohr{(R,+)}\right)^{n-1}.\]
	
	We now turn to the group $\T_n(R)$ for $R$ containing an infinite field $K$. 
 First, we need to show the following strengthening of Proposition \ref{proposition: one step} whose proof is less elementary, as it uses Proposition \ref{fact:char3}.

	\begin{proposition}\label{proposition: tn for polynomial ring}
		Let $R$ be any ring containing an infinite field $K$ (e.g. $R=K[\bar X]$), equipped with any structure. Then $\RR^* \cdot \RRaT_A=\bar R =\RRiD_A$.
	\end{proposition}
	\begin{proof}
Without loss of generality, we can assume that $A =R$ (enlarging $R$ and $A$ if necessary).

We need to check that the right hand side of item (iii) from Lemma \ref{very basic lemma2} is satisfied. For this, take any symmetric $R$-definable set $\bar D$ containing $\RRaT_A$. Note that $\bar D$ is thick in $\RR$ and that its realization $D$ in $R$ is also thick in $R$. We need to show that $R =R^*D$. 

Choose a basis $\{b_i\}_{i \in I}$ for $R$ treated as a linear space over $K$.
\begin{claim2}
For every finite $J \subseteq I$ there is a thick subset $D_J$ of $K$ such that $\sum_{j \in J} D_Jb_{j} \subseteq D$.
\end{claim2}

\begin{proof}[Proof of Claim.]		
By Proposition  \ref{fact:char3} and compactness, there are finitely many $R$-definable characters $\chi_0, \dots,\chi_{k-1} \colon (R,+) \to S^1$ and $m \in \mathbb{N}_{>0}$ such that $\chi_0^{-1}[[-\frac{1}{m}, \frac{1}{m}]] \cap \dots \cap \chi_{k-1}^{-1}[[-\frac{1}{m}, \frac{1}{m}]]\subseteq D$. Let $\chi_{ij} \colon (K,+) \to S^1$ be the character defined as the composition $\chi_i \circ e_j$, where $e_j \colon (K,+) \to R$ is given by $e_j(a) :=ab_j$. Put $n:=|J|$. Then, for $D_{ij}:=\chi_{ij}^{-1}[[-\frac{1}{mn}, \frac{1}{mn}]]$ (where $i=0,\dots,k-1$ and $j \in J$) we have 
$$\chi_i\left[\sum_{j \in J} D_{ij}b_j\right] = \sum_{j \in J}\chi_{ij}\left[D_{ij}\right] \subseteq \left[-\frac{1}{m}, \frac{1}{m}\right].$$
Hence, 
$$\sum_{j \in J} D_{ij}b_j \subseteq \chi_i^{-1}\left[\left[ -\frac{1}{m}, \frac{1}{m}\right]\right].$$
Since each $D_{ij}$ is thick (as the preimage  of a thick set by a homomorphism), the set $D_J$ defined as $\bigcap_{i<k,j \in J} D_{ij}$ is also thick (see \cite[Lemma 1.2]{gismatullin2010absolute}). By the last displayed formula, we have 
\begin{equation*}\sum_{j \in J} D_Jb_j \subseteq \bigcap_{i<k} \chi_i^{-1}\left[\left[ -\frac{1}{m}, \frac{1}{m}\right]\right] \subseteq D. \qedhere \end{equation*}
\end{proof}


\begin{claim2}
	For every $n \in \omega$ and thick subset $D_n$ of $K$ we have $K\cdot D_n^{\times n} = K^{\times n}$, where $X^{\times n}$ denotes the $n$-fold Cartesian power, and $\cdot $ coordinatewise multiplication.
\end{claim2}

\begin{proof}[Proof of Claim.]
	We need to show that for every $a_0,\dots, a_{n-1} \in K$  there exists $a \in K$ such that $(a_0,\dots,a_{n-1}) \in a \cdot D_n^{\times n}$. Since $0 \in D_n$, we can assume that all the $a_i$'s are non-zero. Then the last statement is equivalent to the condition $a_0^{-1}D_n \cap \dots \cap a_{n-1}^{-1}D_n \ne \{ 0\}$. 
	Now, since $D_n$ is thick and each $a_i^{-1} \cdot$ is an automorphism of $(K,+)$, each $a_i^{-1}D_n$ is thick, so the intersection of all of them is also thick by \cite[Lemma 1.2]{gismatullin2010absolute}, so contains a non-zero element, because $K$ is infinite.
\end{proof}

Now take any $x \in R$. Then there are a finite $J \subseteq I$ and $k_j \in K$ for $j \in J$ such that $x = \sum_{j \in J}k_j b_j$. By the first claim, there is a thick subset $D_J$ of $K$ such that $\sum_{j \in J}D_J b_j \subseteq D$. Hence, we have $x \in \sum_{j \in J}K b_j = K\sum_{j \in J}D_J b_j \subseteq KD$, where the equality $\sum_{j \in J}K b_j = K\sum_{j \in J}D_J b_j$ is provided by the second claim. Thus,  we have shown that $R \subseteq KD \subseteq R^*D$, so $R=R^*D$.
\end{proof}

	From Proposition \ref{proposition: tn for polynomial ring}, we obtain that $I_i' = \RR$ for all $i$. So, by Proposition \ref{prop:tn00quot} or Corollary \ref{cor:utn00quot with dagger}, $\dBohr{\T_n(R)}$ is 
		\[
		\begin{pmatrix}
		\dBohr{(R^*,\cdot)} & 0 & 0 & \ldots & 0 & 0 \\
		0 & \dBohr{(R^*,\cdot)} & 0 & \ldots & 0 & 0 \\
		0 & 0 & \dBohr{(R^*,\cdot)} & \ldots & 0 & 0 \\
		\vdots & \vdots & \vdots & \ddots & \vdots & \vdots \\
		0 & 0 & 0 & \ldots & \dBohr{(R^*,\cdot)} & 0 \\
		0 & 0 & 0 & \ldots & 0 & \dBohr{(R^*,\cdot)} \\
		\end{pmatrix} \cong \left(\dBohr{(R^*,\cdot)}\right)^{n}.\]

\subsection{Topological triangular groups}\label{subsection:topological}

We describe how our approach can be adapted to compute the classical Bohr compactification of $\UT_n(R)$ and $\T_n(R)$ treated as topological groups with the product  topology induced from the topology on $R$, where we assume that $R$ is a (unital) topological ring.

In order to do that, we need first to recall how to present model-theoretically the Bohr compactification of a topological group. So let $G$ be a topological group 0-definable in a first order structure $M$ in such a way that all open subsets of $G$ are 0-definable (e.g. we can work in $\Lsetp{M}$). Following \cite[Definition 2.3]{krupil}, we define $\bar G^{00}_{\topo}$ to be the smallest bounded index subgroup of $\bar G$ which is an intersection of some sets of the form $\bar U$ for $U$ open in $G$. Let $\mu$ denote the intersection of the $\bar U$'s for $U$ ranging over all open neighborhoods of the neutral element of $G$; $\mu$ is the group of infinitesimal elements of $\bar G$. Proposition 2.1 of \cite{GPP} or Fact 2.4 of \cite{krupil} says that  $\bar G^{00}_{\topo}$ is a normal subgroup of $\bar G$, and 
the quotient mapping $\pi \colon G \to \bar G/\bar G^{00}_{\topo}$ is the Bohr compactification of $G$ (treated as a topological group). Proposition 2.5 of \cite{krupil} describes $\bar G^{00}_{\topo}$ as the smallest $M$-type-definable [or 0-type-definable], bounded index subgroup of $\bar G$ which contains $\mu$. We will be using this description rather than the original definition.

We need the following variant of Lemma \ref{lem:twosemi}.

\begin{lemma}\phantomsection\label{lem: properties of topological components}
\begin{enumerate}[(i)]
\item Let $I_1,\dots,I_{n}$ be topological groups 0-definable in $M$. Equip $G: = I_n \times \dots \times I_{1}$ with the product topology, and assume that all open subsets of $G$ are  0-definable in $M$. Then $\mu = \mu_{I_n} \times \dots \times \mu_{I_{1}}$ (where $\mu_{I_i}$ is the group of infinitesimals in $I_i$), and $\bar G^{00}_{\topo} = {{}\bar I_{n}}^{00}_{\topo} \times \dots \times {{}\bar I_{1}}^{00}_{\topo}$.
\item Let $K$, $H$, and $N$ be topological groups $0$-definable in $M$, and let  $\phi_1$ and $\phi_2$ be continuous, $0$-definable, respectively left and right actions by automorphisms of $K$ on $N$ and of $H$ on $N$. Equip $G: = (K, H) \ltimes_{\phi_1}^{\phi_2} N$ with the product topology, and assume that all open subsets of $G$ are 0-definable in $M$. Then $\mu = (\mu_K \times \mu_H)  \ltimes_{\phi_1}^{\phi_2} \mu_N$ (where $\mu_K$, $\mu_H$, and $\mu_N$ are the groups of infinitesimals in $K$, $H$, and $N$, respectively), and $\G^{00}_{\topo} = (\bar{K}^{00}_{\topo} \times \bar{H}^{00}_{\topo}) \ltimes_{\phi_1}^{\phi_2} N'$, where $N'$ is the smallest $0$-type-definable, bounded index subgroup of $\bar{N}$ containing $\mu_N$ (equivalently, containing $\bar N^{00}_{\topo}$), and invariant under the actions of both $\bar{K}$ and $\bar{H}$.
\item Let $H$ and $N$ be topological groups 0-definable in $M$, and let $\phi$ be a continuous,  0-definable left action of $H$ on $N$ by automorphisms. 
Equip $G: = H \ltimes_\phi N$ with the product topology, and assume that all open subsets of $G$ are 0-definable in $M$. Then $\mu = \mu_H  \ltimes_\phi \mu_N$ (where $\mu_H$ and $\mu_N$ are the groups of infinitesimals in $H$ and $N$, respectively), and $\bar G^{00}_{\topo} = \bar H^{00}_{\topo} \ltimes_\phi N'$, where $N'$ is the smallest 0-type-definable, bounded index subgroup of $\bar{N}$ containing $\mu_N$ (equivalently, containing $\bar N^{00}_{\topo}$), and invariant under the action of $\bar{H}$.
\end{enumerate}
\end{lemma}

\begin{proof}
(i) follows easily from the definitions of infinitesimals and product topology, and the aforementioned characterization of $\bar G^{00}_{\topo}$ in terms of $\mu$.

(ii)  Note that $G$ is a topological group. Observe that all open subsets of $K$, of $H$, and of $N$ are 0-definable in $M$, so the objects $\mu_K$, $\mu_H$, $\mu_N$, $\bar K^{00}_{\topo}$, $\bar H^{00}_{\topo}$, and $\bar N^{00}_{\topo}$ are defined. As in (i), the equality $\mu = (\mu_K \times \mu_H)  \ltimes_{\phi_1}^{\phi_2} \mu_N$ is clear from definitions. Having this, the equality $\G^{00}_{\topo} = (\bar{K}^{00}_{\topo} \times \bar{H}^{00}_{\topo}) \ltimes_{\phi_1}^{\phi_2} N'$ follows as in Lemma \ref{lem:twosemi}, using the aforementioned characterization of $\bar G^{00}_{\topo}$ in terms of $\mu$.

(iii) follows from (ii).
\end{proof}

Let now $R$ be a (unital) topological ring. We work in $\Lsetp{R}$. Define a sequence $I_{i}(\RR)$, $i > 0$, of 0-type-definable subgroups of $\RRa$ as follows: $I_{1}(\RR) := \RRaT_{\topo}$, and for $i > 0$, $I_{i+1}(\RR)$ is the smallest 0-type-definable subgroup of $\RRa$ containing the set $\RR \cdot I_{i}(\RR)$. By Corollary \ref{corollary: charact. of R^00_topo}, we have
	\[\RRaT_{\topo} = I_{1}(\RR) \leq I_{2}(\RR) \leq \ldots \leq I_{i}(\RR) \leq \ldots \leq \RRiT_{\topo},\]
and all the comments right after the definition of $I_{i,A}$ in Subsection \ref{sec:unig} have their obvious counterparts. In particular, $I_j$ is constant for $j \geq i$ if and only of $I_i = \RRiT_{\topo}$.

Using Lemma \ref{lem: properties of topological components}, one can easily check that the proof of Lemma \ref{lem:utninv}(i) adapts to the present context, so we get the following variant of Proposition \ref{prop:utn00}.

	\begin{proposition} \label{prop:utn00 in topological case}
		Let $R$ be a (unital) topological ring. Then
		\[\UT_n(\RR)^{00}_{\topo} =
		\begin{pmatrix}
		1 & I_1 & I_2 & \ldots & I_{n-2} & I_{n-1} \\
		0 & 1 & I_1 & \ldots & I_{n-3} & I_{n-2} \\
		0 & 0 & 1 & \ldots & I_{n-4} & I_{n-3} \\
		\vdots & \vdots & \vdots & \ddots & \vdots & \vdots \\
		0 & 0 & 0 & \ldots & 1 & I_1 \\
		0 & 0 & 0 & \ldots & 0 & 1 \\
		\end{pmatrix},\]
		where $I_i = I_i(\RR)$.
	\end{proposition}

Keeping in mind the identifications as in the discrete case described right after	 Proposition \ref{prop:utn00},
the proof of the next result is the same as for Proposition  \ref{prop:utn00quot} (using Proposition \ref{prop:utn00 in topological case} in place of \ref{prop:utn00}).

\begin{proposition}
		\label{prop:utn00quot in topological case}
		Let $R$ be a (unital) topological ring. Then the Bohr compactification of the topological group $\UT_n(R)$ equals
		\[\UT_n(\RR)/\UT_n(\RR)^{00}_{\topo}  \cong
		\begin{pmatrix}
		1 & B/I_1 & B/I_2 & \ldots & B/I_{n-2} & B/I_{n-1} \\
		0 & 1 & B/I_1 & \ldots & B/I_{n-3} & B/I_{n-2} \\
		0 & 0 & 1 & \ldots & B/I_{n-4} & B/I_{n-3} \\
		\vdots & \vdots & \vdots & \ddots & \vdots & \vdots \\
		0 & 0 & 0 & \ldots & 1 & B/I_1 \\
		0 & 0 & 0 & \ldots & 0 & 1 \\
		\end{pmatrix},\]
		where $B := \RRa$ and the topology on the right hand side is the product topology induced from the logic topologies on the quotients $B/I_i$. The quotient $B/I_1$ is exactly the Bohr compactification of the topological group $(R,+)$.
	\end{proposition}

In order to state similar results for $\T_n(\RR)$, we need to and do assume that the group of units $(R^*,\cdot)$ is topological with the topology induced from $R$.
	As in the discrete case, to state the results for the group $\T_n(\RR)$, we need to define another non-decreasing sequence $I_i'(\RR)$, $i\in\N_{> 0}$, of $0$-type-definable subgroups of $\RRa$ as follows: $I_1'(\RR)$ is the smallest $0$-type-definable subgroup of $\RRa$ which contains $\RRaT_{\topo}$ and which is closed under multiplication by $\RR^*$ from both left and right; for $i > 0$, $I_{i+1}'(\RR)$ is the smallest $0$-type-definable subgroup of $\RRa$ that contains the set $\RR \cdot I_i'(\RR) \cdot \RR^*$ and that is closed under multiplication by $\RR^*$ from both left and right.
And again, the comments right after the definition of $I_{i,A}'$ in Subsection  \ref{sec:unig} have their obvious counterparts. In particular, 
$I_i(\RR) \subseteq I_i'(\RR) \subseteq  \RRiT_{\topo}$ for all $i$. 

Using Lemma \ref{lem: properties of topological components}, one can easily check that the proof of Lemma \ref{lem:utninv}(ii) adapts to the present context, so we get the following variants of Propositions \ref{prop:tn00} and \ref{prop:tn00quot}.

	\begin{proposition}
		\label{prop:tn00topo}
		\[\T_n(\RR)^{00}_{\topo}  =
		\begin{pmatrix}
		(\RR^*,\cdot)_{\topo}^{00} & I_1' & I_2' & \ldots & I_{n-2}' & I_{n-1}' \\
		0 & (\RR^*,\cdot)_{\topo}^{00} & I_1' & \ldots & I_{n-3}' & I_{n-2}' \\
		0 & 0 & (\RR^*,\cdot)_{\topo}^{00} & \ldots & I_{n-4}' & I_{n-3}' \\
		\vdots & \vdots & \vdots & \ddots & \vdots & \vdots \\
		0 & 0 & 0 & \ldots & (\RR^*,\cdot)_{\topo}^{00} & I_1' \\
		0 & 0 & 0 & \ldots & 0 & (\RR^*,\cdot)_{\topo}^{00} \\
		\end{pmatrix}.\]
	\end{proposition}

	The group operation in the result below uses the identifications analogous to those discussed before Proposition \ref{prop:utn00quot}:
	
	\begin{proposition}
		\label{prop:tn00quottopo}
		The Bohr compactification of the topological group $\T_n(R)$ is
		\[\T_n(\RR)/\T_n(\RR)^{00}_{\topo} \cong
		\begin{pmatrix}
		P & B/I_1' & B/I_2' & \ldots & B/I_{n-2}' & B/I_{n-1}' \\
		0 & P & B/I_1' & \ldots & B/I_{n-3}' & B/I_{n-2}' \\
		0 & 0 & P & \ldots & B/I_{n-4}' & B/I_{n-3}' \\
		\vdots & \vdots & \vdots & \ddots & \vdots & \vdots \\
		0 & 0 & 0 & \ldots & P & B/I_1' \\
		0 & 0 & 0 & \ldots & 0 & P \\
		\end{pmatrix},\]
		where $P: = (\RR^*,\cdot)/(\RR^*,\cdot)_{\topo}^{00}$ is the Bohr compactification of the topological group $(R^*,\cdot)$, $B: = \RRa$, and $\cong$ is a topological group isomorphism, with the right hand side equipped with the product topology induced from the logic topologies on the quotients $B/I_i'$.
	\end{proposition}	

We expect that the following variant of $(\dagger)$ is true for every topological ring:
	\[I_i(\RR) = \RRiT_{\topo} \text{ for all } i \geq 2.\eqno{(\dagger\dagger)}\]
Using Lemma \ref{lemma: mu + 00 is an ideal}, one can show that $(\dagger\dagger)$ would follow from a positive answer to Question \ref{question: to be answered in the next paper}. This will be discussed in the forthcoming paper. For now, notice that if $R$ satisfies $(\dagger\dagger)$, then our formulas for the Bohr compactifications of the topological groups $\UT_n(R)$ and $\T_n(R)$ obtained in Propositions \ref{prop:utn00quot in topological case} and \ref{prop:tn00quottopo} simplify in the same manner as in Corollary \ref{cor:utn00quot with dagger} but with each definable Bohr compactification replaced by the (topological) Bohr compactification.

	\begin{example}\label{example: top. field}
		Let $R = K$ be a topological field (e.g. $\R$). Then $\bar{K} \cdot (\bar{K}, +)^{00}_{\topo} = \bar{K}$, and so for all $i > 1$, $I_{i}(\bar{K}) = \bar{K}$. Let $Q := (\bar{K},+)/(\bar{K}, +)^{00}_{\topo}$, i.e. the Bohr compactification of the topological group $(K,+)$. By Proposition \ref{prop:utn00quot in topological case}, the Bohr compactification of the topological group $\UT_n(K)$ is
		\[
		\UT_n(\bar{K})/\UT_n(\bar{K})^{00}_{\topo} \cong
		\begin{pmatrix}
		1 & Q & 0 & \ldots & 0 & 0 \\
		0 & 1 & Q & \ldots & 0 & 0 \\
		0 & 0 & 1 & \ldots & 0 & 0 \\
		\vdots & \vdots & \vdots & \ddots & \vdots & \vdots \\
		0 & 0 & 0 & \ldots & 1 & Q \\
		0 & 0 & 0 & \ldots & 0 & 1 \\
		\end{pmatrix} \cong Q^{n-1}.\]
Let $P:= (\bar K^*,\cdot)/(\bar K^*, \cdot)^{00}_{\topo}$, i.e. the Bohr compactification of the topological group $(K^*,\cdot)$.
Since  $\bar{K}^* \cdot (\bar{K}, +)^{00}_{\topo} = \bar{K}$, and so for all $i \geq 1$, $I_{i}'(\bar{K}) = \bar{K}$, by Proposition \ref{prop:tn00quottopo}, the Bohr compactification of the topological group $\T_n(K)$ is
\[ \T_n(\bar{K})/\T_n(\bar{K})^{00}_{\topo} \cong
\begin{pmatrix}
		P & 0 & 0 & \ldots & 0 & 0 \\
		0 & P & 0 & \ldots & 0 & 0 \\
		0 & 0 & P & \ldots & 0 & 0 \\
		\vdots & \vdots & \vdots & \ddots & \vdots & \vdots \\
		0 & 0 & 0 & \ldots & P & 0 \\
		0 & 0 & 0 & \ldots & 0 & P \\
		\end{pmatrix} \cong P^{n}.\]
\end{example}


	\section*{Acknowledgements}
	We would like to thank the anonymous referee for careful reading of the original manuscript and all the suggestions.

	\bibliographystyle{alpha-abbrvsort}
	\bibliography{bohr9}
\end{document}